\documentclass[a4paper,11pt,reqno]{amsart}

\usepackage{latexsym}
\usepackage[english]{babel}
\usepackage{fancyhdr}
\usepackage[mathscr]{eucal}
\usepackage{amsmath}
\usepackage{mathrsfs}
\usepackage{amsthm}
\usepackage{amsfonts}
\usepackage{amssymb}
\usepackage{amscd}
\usepackage{bbm}
\usepackage{graphicx}
\usepackage{graphics}
\usepackage{latexsym}
\usepackage{color}
\usepackage{calc}
\usepackage{dsfont}
\usepackage{mathtools}

\usepackage{geometry}
\usepackage[hidelinks]{hyperref}
\usepackage[utf8]{inputenc}

\newcommand{\ud}{\mathrm{d}}

\newcommand{\ii}{\mathrm{i}}

\newcommand{\R}{\mathbb{R}}
\newcommand{\Rd}{\mathbb{R}^{2}}
\newcommand{\G}{\mathcal{G}}

\newcommand{\la}{\lambda}
\newcommand{\ep}{\varepsilon}
\newcommand{\D}{\mathcal{D}}
\newcommand{\F}{\mathcal{F}}

\newcommand{\C}{\mathbb C}

\newcommand{\Eps}{\mathcal{E}}
\newcommand{\na}{\nabla}
\newcommand{\lap}{\Delta}

\newcommand{\deb}{\rightharpoonup}

\newcommand{\f}[2]{\frac{#1}{#2}}
\newcommand{\scal}[2]{\langle #1,#2\rangle}

\theoremstyle{plain}
\newtheorem{theorem}{Theorem}[section]
\newtheorem{lemma}[theorem]{Lemma}
\newtheorem{corollary}[theorem]{Corollary}
\newtheorem{proposition}[theorem]{Proposition}

\theoremstyle{definition}
\newtheorem{definition}[theorem]{Definition}

\newtheorem{remark}[theorem]{Remark}
\newtheorem*{remark*}{Remark}

\begin{document}


\title[2D NLS ground states with Coulomb potential and point interaction]
{Two dimensional NLS ground states with attractive Coulomb potential and point interaction}

\author[F. Boni]{Filippo Boni}
\address[F. Boni]{Scuola Superiore Meridionale, Largo S. Marcellino, 10, 80138, Napoli, Italy.}
\email{f.boni@ssmeridionale.it}

\author[M.~Gallone]{Matteo Gallone}
\address[M.~Gallone]{Scuola Internazionale Superiore di Studi Avanzati (SISSA), via Bonomea, 265, 34136 Trieste Italy}
\email{matteo.gallone@sissa.it}

\begin{abstract}
	We study the existence and the properties of ground states at fixed mass for a focusing nonlinear Schr\"odinger equation in dimension two with a point interaction, an attractive Coulomb potential and a nonlinearity of power type. We prove that for any negative value of the Coulomb charge, for any positive value of the mass and for any $L^{2}$-subcritical power nonlinearity, such ground states exist and exhibit a logarithmic singularity where the interaction is placed. Moreover, up to multiplication by a phase factor, they are positive, radially symmetric and decreasing. An analogous result is obtained also for minimizers of the action restricted to the Nehari manifold, getting the existence also in the $L^{2}$-critical and supercritical cases.
\end{abstract}

\date{\today}


\thanks{We thank R. Adami, B. Langella and D. Noja for many interesting and fruitful discussions on the project. The work was partially supported by  the INdAM Gnampa 2022 project ``Modelli matematici con singolarità per fenomeni di interazione''. M.G. received financial support from the MIUR-PRIN 2017 project MaQuMA cod. 2017ASFLJR, the European Research Council (ERC) under the European Union’s Horizon 2020 research and innovation program ERC StG MaMBoQ, n.802901. We also thank GNFM, the Italian National Group of Mathematical Physics. M.G. acknowledges Dipartimento di
Matematica ``G. Lagrange'', Politecnico di Torino, where part of this work has been carried on.}

\maketitle



\section{Introduction}
We consider a nonlinear Schr\"odinger equation in dimension two in presence of a point interaction and an attractive Coulomb potential, that can be formally described by the equation
\begin{equation}\label{eq:symm_op}
\ii\partial_{t}\psi=(-\Delta + \nu|x|^{-1} + \alpha \delta_0(x))\psi-|\psi|^{p-2}\psi,\quad \nu<0 \, , \quad p>2 \, .
\end{equation}

There are several reasons for which equation \eqref{eq:symm_op} is relevant for the applications in various contexts of physics and also for the mathematical techniques that have been developed for its study.

From the physical point of view, nonlinear Schr\"odinger equations in dimension two describe, among other phenomena, the dynamics of Bose-Einstein condensates \cite{KayKirkpatrick2011,Jeblick2019,Li2021}. In this context the nonlinear term describes the self-interaction of the ``quantum field'' $\psi$; the Coulomb potential describes the interaction with an external electrostatic potential and the point interaction models the interaction between the field and the source of the potential.

Point interactions, corresponding to the $\delta$ term in \eqref{eq:symm_op}, arise naturally in quantum mechanics when considering the ``Darwin correction'' to the hydrogen atom. This is an effective correction, introduced with heuristic arguments by physicists \cite{Bransden2003-yf}, taking into account an effective smearing out of the wave function of the quantum system due to the ``zitterbewegung'', the phenomenon of fast oscillations of the electron close to the nucleus.

From the mathematical point of view, nonlinear models with point interactions have been studied first in dimension one, and only recently in dimension two and three. In dimension one there are several different types of point interactions, since the Laplacian as a densely defined and symmetric operator $C^\infty_0(\mathbb{R}\setminus\{0\})\subset L^2(\mathbb{R})$ admits a four-parameter family of boundary conditions. Some of them have been considered in the nonlinear setting: delta \cite{Cao-Malomed-95,Fuk-Ohta-Ozawa-08,Holm-Marz-Zwo-07}, delta prime \cite{Adami-Noja-Visc-13} and F\"{u}l\"{o}p Tsutsui conditions \cite{Adami-Rev,Flp-00}. In the two and three dimensional case, where the deficiency index of the Laplacian on $C^\infty_0(\mathbb{R}^d\setminus\{0\}) \subset L^2(\mathbb{R}^d)$ is one, there is only one type of point interaction. In recent years some results on nonlinear Schr\"odinger equations in presence of point interactions in dimension larger than one have been obtained, concerning existence and stability of ground states \cite{ABCT-3d,Adami-Boni-Carlone-Tentarelli-2022,FGI-22,GMS-22}, global well–posedness \cite{Cac-Finco-Noja-21}, blow–up phenomena \cite{Finco-noja-22} and scattering \cite{Cac-Finco-Noja-22}.

Nonlinear Schr\"odinger equation with a multiplicative Coulomb potential has been also considered in the past years, with the aim of studying the existence, uniqueness and stability of positive solutions \cite{Benguria-86, Dinh-2021,Li-20} and global well-posedness, blow up and scattering \cite{Miao-22}. 

The present paper fits in this line of research, proposing a model that combines a point interaction and an attractive Coulomb potential. For this problem, we prove the existence of ground states and we study their properties. We prove that, up to multiplication of a phase, ground states are positive, rotationally symmetric and radially decreasing.

The existence of ground states for any value of the mass, as well as their properties are expected results.
Indeed, the same results hold both for the NLS with delta interaction \cite{Adami-Boni-Carlone-Tentarelli-2022} and for the NLS with external attractive Coulomb potential \cite{Dinh-2021}. Nevertheless, this does not diminish their relevance, since the simultaneous presence of a point interaction and a Coulomb potential makes the problem technically harder due to the form of the Green's function \eqref{G-lanu}, especially in the proof of inequalities of the Gagliardo-Nirenberg type (see Lemma \ref{lem:GNnu<0}) or in the proof of Theorem \ref{ex-gs-mu}.

Furthermore, it would be also interesting to investigate whether such ground state solutions are orbitally stable or not, but, in order to obtain similar results by using \cite{GSS-87} or \cite{Caze-Lions-82}, a well-posedness result is needed, and this is not straightforward since it requires Strichartz estimates for the operator $-\Delta+\nu|x|^{-1}+\alpha \delta_0$.

Finally, it would be interesting to extend the analysis of ground states in presence of a repulsive Coulomb potential, i.e. when $\nu>0$. In such a situation, the competition between the point interaction, that in two and three dimensions can be considered always ``attractive'' (see \cite{Adami-Boni-Carlone-Tentarelli-2022} and \cite{ABCT-3d}), and the repulsive Coulomb potential could affect both the existence results and the qualitative properties of ground states, but new techniques are needed to study this problem.

%
%
%
%
%

\subsection*{Notation.} 
In the following, we will denote $\|u\|_{p}:=\|u\|_{L^{p}(\R^{2})}$, for $p>1$.

\subsection{Setting and main results}
In this section, we define rigorously the problem and state the main results of the paper. 

Let us stress that the expression \eqref{eq:symm_op} is formal and a rigorous definition to $-\Delta + \nu|x|^{-1} + \alpha \delta_0(x)$ can be given through the theory of self-adjoint extensions of symmetric operators. In particular, we are interested in the self-adjoint realisations in $L^2(\mathbb{R}^2)$ of the densely defined symmetric operator
\begin{equation}\label{eq:sym_vero}
	H_{\nu}\;=\; - \Delta + \nu|x|^{-1} \, \qquad \text{on} \qquad \mathcal{D}(H_{\nu})\;=\; C_0^\infty(\mathbb{R}^2 \setminus\{0\}) \, .
\end{equation}

The problem of self-adjoint realisations for this model, as well as for the related one- and three-dimensional models has attained a certain interest in the literature in the last decade (see e.g. \cite{2010-Duclos-2D-Delta-Coulomb,2009-deOliveira-Verri}), where the classification of self-adjoint operators was obtained using von Neumann extension theory. Since we are ultimately interested in the applications to non-linear problems and the approach will be mainly variational, we need the classification of the associated quadratic forms, the latter being naturally obtained in the framework of the Kre\u{\i}n-Vi\v{s}ik-Birman extension theory (see \cite{GMO-KVB2017} for a recent review or \cite{Gallone-Michelangeli-Book} for a comprehensive treatment). It is thus natural to re-obtain the results of the above works in this context, and to complement with the classification of quadratic forms (see Appendix \ref{app:self-adj} for a detailed discussion of this topic). 

Let us define for $\lambda>0$ and $\nu>-\sqrt{\la}$ the Green's function of $-\Delta+\f{\nu}{|x|}+\la$
\begin{equation}
\label{G-lanu}
	\G_{\lambda,\nu}(|x|) \;:=\; \frac{\Gamma\left({\textstyle \frac{1}{2}+\frac{\nu}{2 \sqrt{\lambda}}} \right)}{2 \pi} e^{-\sqrt{\lambda} |x|}  U_{\frac{1}{2}+\frac{\nu}{2 \sqrt{\lambda}},1}\left(2 \sqrt{\lambda} |x|\right),
\end{equation}
and the quantity
\begin{equation}\label{eq:thetaDef}
	\theta_{\lambda,\nu} \;:=\;\frac{1}{2 \pi} \left(\psi\left({\textstyle \frac{1}{2}+\frac{\nu}{2 \sqrt{\lambda}}}\right)+ 2 \gamma + \ln( 2 \sqrt{\lambda}) \right)\bigg) \,,
\end{equation}
where $U_{a,b}(z)$ is the \emph{Tricomi function} defined in \cite[Eq. 13.1.6]{Abramowitz-Stegun-1964} and $\psi$ denotes the digamma function, that is $\psi(x)=\Gamma'(x)/\Gamma(x)$ with $\Gamma$ being the gamma function. 

The next proposition collects some results about the operator and the quadratic form of the self-adjoint realisations of \eqref{eq:sym_vero}.


\begin{proposition}~\label{prop:Classification} For any $\nu\in \R$, the symmetric operator $H_{\nu}$ defined in \eqref{eq:sym_vero} admits a one-parameter family of self-adjoint extensions $\{H_{\nu,\alpha}\}_{\alpha \in \mathbb{R} \cup \{\infty\}}$. For any $\alpha\in \R$:
	\begin{itemize}
	\item[(i)] The domain of self-adjointness is
	\begin{equation}\label{eq:Domain2DExtension-OP}
			\mathcal{D}(H_{\nu,\alpha}) \;=\;\left \{ u \in \mathcal{D}(H^*) \, | \, u=\phi_{\la}+q \G_{\lambda,\nu} \, , \, \phi_{\la} \in\mathcal{D}(H_F), \phi_{\la}(0)=q ( \alpha+ \theta_{\lambda,\nu}) \right\},
\end{equation}
	with $H^{*}$ and $H_F$ being the adjoint and the Friedrichs extension of $H_{\nu}$.
	\item[(ii)] The domain of the quadratic form is given by
	\begin{equation}
	\label{dom-form}
		\D_{\nu}:=\mathcal{D}(Q_{\nu}) \;=\; H^1(\mathbb{R}^2) \;+\; \mathrm{span}_{\mathbb{C}}\{\G_{\lambda,\nu}\}
	\end{equation}
	and the action on $u=\phi_\lambda+q \G_{\lambda,\nu}$  by
	\begin{equation}\label{eq:FormValue}
		Q_{\nu}(u) \;=\;  \Vert \nabla \phi_\lambda \Vert_2^2+\nu \Vert |x|^{-1/2} \phi_\lambda \Vert_2^2 + \lambda (\Vert \phi_\lambda \Vert_2^2-\Vert u \Vert_2^2)+ (\alpha+\theta_{\lambda,\nu})|q|^2
	\end{equation}
	\item[(iii)] If $\nu<0$, then $H_\nu$ has infinitely many negative eigenvalues accumulating at $0$. Moreover, for any $\alpha \in \mathbb{R}$, there exists only one negative eigenvalue $-\omega_\nu$ of $H_{\nu,\alpha}$ such that $-\omega_{\nu}<-\nu^{2}$. $\omega_\nu$ is the unique solution in the variable $\la>\nu^{2}$ of the equation 
	\begin{equation*}
	\alpha+\theta_{\la,\nu}=0
	\end{equation*}
satisfying 
	\begin{equation*}
	-\omega_{\nu}=\inf_{v\in \D_{\nu}\setminus\{0\}}\f{Q_{\nu}(v)}{\|v\|_{2}^{2}}<-\nu^{2}=\inf_{v\in H^{1}(\R^{2})\setminus\{0\}}\f{Q_{\nu}(v)}{\|v\|_{2}^{2}}.
	\end{equation*}
		\end{itemize}	
\end{proposition}



The main goal of the present paper is to investigate the existence and the properties of ground states at fixed mass of the nonlinear energy
\begin{equation}
\label{en-fan}
F_{\nu}(v):=\f{1}{2}Q_{\nu}(v)-\f{1}{p}\|v\|_{p}^{p},
\end{equation}
when the Coulomb potential is attractive and the nonlinear term is $L^{2}$-subcritical, i.e. when $\nu<0$ and $2<p<4$. 
\begin{definition}
Given $\mu>0$, a function $u$ belonging to the set
\begin{equation}
\label{const-L2}
\D_{\nu}^{\mu}:=\left\{v\in \D_{\nu}\,:\, \|v\|_{2}^{2}=\mu\right\}
\end{equation}
and satisfying
\begin{equation*}
F_{\nu}(u)=\F_{\nu}(\mu):=\inf_{v\in \D_{\nu}^{\mu}}F_{\nu}(v)
\end{equation*}
is called a \emph{ground state of \eqref{en-fan} at mass $\mu$}.
\end{definition}

As usual, ground states of \eqref{en-fan} at mass $\mu$ are bound states, i.e. they satisfy
\begin{align}
\label{bcond}
u\in D(H_{\nu,\alpha}),\\
\label{EL-eq}
H_{\nu,\alpha}(u)-|u|^{p-2}u+\omega u=0,
\end{align}
for some $\omega\in \R$ (see Appendix \ref{app-gbstates}): in this case, $\omega$ plays the role of a Lagrangian multiplier.

An alternative approach to prove the existence of bound states consists in fixing $\omega\in\R$ and looking for minimizers of the action functional $S_{\nu}^{\omega}:\D_{\nu}\to \R$, defined as
\begin{equation}
\label{S-nu-om}
S_{\nu}^{\omega}(v):=F_{\nu}(v)+\f{\omega}{2}\|v\|_{2}^{2},
\end{equation}
restricted to the so-called Nehari Manifold, that is the zero-level set of the functional $I_{\nu}^{\omega}:D_{\nu}\to \R$, defined as 
 \begin{equation}
\label{I-nu-om}
I_{\nu}^{\omega}(v):=\scal{\left(S_{\nu}^{\omega}\right)'(v)}{v}=Q_{\nu}(v)+\omega\|v\|_{2}^{2}-\|v\|_{p}^{p}.
\end{equation}

\begin{definition}
Given $\omega\in \R$, a function $u$ belonging to the Nehari manifold
\begin{equation}
\label{Neh-man}
N_{\nu}^{\omega}:=\{v\in \D_{\nu}\setminus\{0\}\,:\,I_{\nu}^{\omega}(v)=0\},
\end{equation}
and satisfying 
\begin{equation}
S_{\nu}^{\omega}(u)=d_{\nu}(\omega):=\inf_{v\in N_{\nu}^{\omega}}S_{\nu}^{\omega}(v)\label{eq:S_inf}
\end{equation}
is a \emph{minimizer of the action \eqref{S-nu-om} at frequency $\omega$}.
\end{definition}
We observe that also minimizers of the action \eqref{S-nu-om} at frequency $\omega$ are bound states, i.e. they satisfy \eqref{bcond}--\eqref{EL-eq} (see again Appendix \ref{app-gbstates}).

We can now state the main results of the paper. The first result concerns the existence and the properties of ground states at fixed mass of \eqref{en-fan} when $2<p<4$ and $\nu<0$.

\begin{theorem}
\label{ex-gs-mu}
Let $2<p<4$, $\nu<0$ and $\alpha\in \R$. Then for every $\mu>0$ ground states of \eqref{en-fan} at mass $\mu$ exist. Moreover, every ground state presents a logarithmic singularity at the origin and is positive, radially symmetric and decreasing along the radial direction, up to gauge invariance.
\end{theorem}

The second result concerns the existence and the characterization of the minimizers of the action at frequency $\omega$.

\begin{theorem}
\label{ex-min-act}
Let $p>2$, $\nu<0$ and $\alpha\in \R$. Then for every $\omega>\omega_{\nu}$ the minimizers of the action \eqref{S-nu-om} at frequency $\omega$ exist. Moreover, every minimizer presents a logarithmic singularity at the origin and is positive, radially symmetric and decreasing, up to gauge invariance.
\end{theorem}

Differently from Theorem \ref{ex-gs-mu}, Theorem \ref{ex-min-act} provides the existence of bound states, i.e. solutions of \eqref{bcond}-\eqref{EL-eq}, also in the $L^{2}$-critical and supercritical case, i.e. for $p\geq 4$.

Moreover, let us observe that the properties of ground states at fixed mass stated in Theorem \ref{ex-gs-mu} are deduced from the properties of the minimizers of the action via Lemma \ref{GSareAM}, in which a connection between ground states at fixed mass and minimizers of the action is established.

\noindent


\section{Preliminaries}

\subsection{Properties of $\theta_{\la,\nu}$ and the Green's function $\G_{\la,\nu}$}

Let $\nu<0$ and $\la>\nu^{2}$. We recall that the asymptotic of $\G_{\lambda,\nu}$  as $|x| \to 0$ is
\begin{equation}
	\G_{\lambda,\nu}(|x|) \;=\;-\frac{1}{2 \pi} \bigg[\ln |x|+\psi\big({\textstyle\frac{1}{2}+\frac{\nu}{2 \sqrt{\lambda}} } \big)+2 \gamma + \ln( 2 \sqrt{\lambda}) + \nu |x| \ln |x|\bigg] + O(|x|)\, .
\end{equation}
and the following integral representation holds \cite[Eq. 13.2.5]{Abramowitz-Stegun-1964}
\begin{equation}
\label{int-Gla}
\G_{\la,\nu}(|x|)=\f{1}{2\pi}e^{-\sqrt{\la}|x|}\int_{0}^{+\infty}e^{-|x|t}t^{\frac{\nu}{2 \sqrt{\lambda}}-\frac{1}{2}}(1+t)^{-\frac{\nu}{2 \sqrt{\lambda}}-\frac{1}{2}}\, \ud t.
\end{equation}

In the next lemma, we give a first estimate of the $L^{q}$ norm of $\G_{\la,\nu}$.
\begin{lemma}
\label{Lrnorm-Green}
Let $s>1$, $\nu<0$ and $\la>\nu^{2}$. Then for any $0<\sigma<\f{2}{s}$ there results
\begin{equation*}
	\Vert \mathcal{G}_{\lambda,\nu} \Vert_{s}^s \leq \frac{C_{\la,\nu}}{\lambda^{1-\frac{\sigma s}{2}}}
\end{equation*}
with $C_{\la,\nu}:=C(\lambda,\nu,\sigma,s)$ given explicitly by
\begin{equation*}
	C_{\la,\nu} \;=\; \f{\pi^{1-s}}{s^{2-\sigma s}}\left[\f{1}{s^{\sigma s}\la^{\f{\sigma s}{2}}}\left(\f{2\sqrt{\la}}{\sqrt{\la}+\nu}\right)^{s}+\sigma^{s(\sigma-1)}e^{-\sigma s}\left(\f{1}{2-\sigma s}+\f{1}{e}\right)\right].
\end{equation*}
\end{lemma}

\begin{proof}
The proof follows estimating \eqref{int-Gla}. First of all, let us observe that
	\[
		\begin{split}
			\int_0^{+\infty} e^{-rt} t^{\frac{\nu}{2 \sqrt\lambda}-\frac12} (1+t)^{-\frac{\nu}{2\sqrt\lambda}-\frac12} \, \ud t \leq \int_0^1 e^{-rt} t^{\frac{\nu}{2\sqrt\lambda}-\frac12} \, \ud t + \int_1^{+\infty}  \frac{e^{-rt}}{t} \, \ud t
		\end{split}
	\]
	The first integral can be estimated using $e^{-rt} \leq 1$, so that
	\[
		\int_0^1 e^{-rt} t^{\frac{\nu}{2\sqrt\lambda}-\frac12} \, \ud t \leq \frac{2 \sqrt{\lambda}}{\sqrt{\lambda}+\nu} \, .
	\]
	In order to bound the second integral, let us observe that, for any $\sigma > 0$,
	\[
		x^{\sigma}e^{-x } \leq \sigma^\sigma e^{-\sigma},\quad x>0,\quad \sigma>0,
	\]
thus
	\[
		\int_1^\infty  \frac{e^{-rt}}{t} \, \ud t \leq \frac{\sigma^\sigma e^{-\sigma}}{r^\sigma} \int_1^{+\infty} t^{-1-\sigma} \, \ud t = \frac{\sigma^{\sigma-1} e^{-\sigma}}{r^\sigma}.
	\]
	
	Since $(a+b)^s \leq 2^{s-1}\left(a^s+b^s\right)$ for every $a,b>0$, we get
	\[
		\begin{split}
			\Vert \mathcal{G}_{\nu,\lambda} \Vert_s^s &\leq \pi^{1-s} \left[\int_0^{+\infty} e^{-\sqrt{\lambda} s r} \left(\frac{2 \sqrt{\lambda}}{\sqrt{\lambda}+\nu} \right)^s r \, \ud r + \int_0^{+\infty} \frac{\sigma^{s(\sigma-1)} e^{-sr}}{r^{s\sigma}}e^{\sqrt{\lambda} s r} r \, \ud r \right] \\
			&= \pi^{1-s}\left[\left(\frac{2 \sqrt{\lambda}}{\sqrt{\lambda}+\nu} \right)^s \int_0^{+\infty} e^{-\sqrt{\lambda} s r} r \, \ud r +  \sigma^{s(\sigma-1)} e^{-\sigma s} \int_0^{+\infty} r^{-s\sigma +1} e^{-\sqrt{\lambda} s r} \, \ud r \right] \\
			&\leq \pi^{1-s} \left[\left(\frac{2 \sqrt{\lambda}}{\sqrt{\lambda}+\nu} \right)^s \frac{1}{\lambda s^2} \int_0^{+\infty} \tau e^{-\tau} \, \ud \tau + \frac{\sigma^{s(\sigma-1)} e^{-\sigma s}}{(s \sqrt{\lambda})^{2-s\sigma}} \int_0^{+\infty} \tau^{-s \sigma+1} e^{-\tau} \, \ud \tau \right].
		\end{split} 
	\]	
	We conclude noting that $\int_{0}^{+\infty}\tau e^{-\tau}\,\ud \tau=1$ and
	\[
		\int_{0}^{+\infty} \tau^{-s \sigma+1} e^{-\tau} \, \ud \tau \leq \int_0^1 \tau^{-s\sigma+1} \, \ud \tau + \int_1^{+\infty} e^{-\tau} \, \ud \tau = \frac{1}{2-s\sigma}+\frac{1}{e} \, .
	\]
\end{proof}

\begin{remark}
We observe that $C_{\la,\nu}$ in Lemma \ref{Lrnorm-Green} satisfies $C_{\la,\nu}\to+\infty$ as $\la\to (\nu^{2})^{+}$, hence the upper bound on $\|\G_{\la,\nu}\|_{s}^{s}$ diverges as $\la\to (\nu^{2})^{+}$. For this reason, in order to obtain an estimate depending only on $\la$ we have to choose $\la\geq C$ for some $C>\nu^{2}$. 
\end{remark}

In particular, the next two lemmas provide an estimate of the $L^{2}$ norm and the $L^{p}$ norm of $\G_{\la,\nu}$, for $2<p<4$,  in terms of $\la$ for every $\la\geq \max\{1,4\nu^{2}\}$. 
\begin{lemma}
\label{Gla2}
Let $\nu<0$. Then for any $\la\geq \max \{1,4\nu^{2}\}$ and $2<p<4$ there results
\begin{equation*}
\|\G_{\la,\nu}\|_{2}^{2}\leq  \f{C_{p}}{\la^{\f{p}{4}}},
\end{equation*}
with 
\begin{equation*}
C_{p}=\frac{1}{2^{\frac{p}{2}} \pi} \left[2^{\f{p+4}{2}}+\left(\f{4}{4-p}\right)^{\f{p}{2}}e^{-\f{4-p}{2}}\left(\f{2}{p}+\f{1}{e}\right)\right].
\end{equation*}
\end{lemma}
\begin{proof}
The proof follows applying Lemma \ref{Lrnorm-Green} with $s=2$ and $\sigma=\f{4-p}{4}<1$.
\end{proof}

Let us point out that, even if not natural, the dependence on $2<p<4$ of the estimate of $\|\G_{\la,\nu}\|_{2}$ in Lemma \ref{Gla2} is functional to the proof of the Gagliardo-Nirenberg type inequality in Lemma \ref{lem:GNnu<0}.
\begin{lemma}
\label{Glap}
Let $2<p<4$ and $\nu<0$. Then for any $\lambda\geq \max\{1,4\nu^{2}\}$ there results
\begin{equation*}
\|\G_{\la,\nu}\|_{p}^{p}\leq \f{\widetilde{C}_{p}}{\la^{\f{p}{4}}},
\end{equation*}
with 
\begin{equation*}
	\widetilde{C}_{p}=\frac{\pi^{1-p}}{p^{\frac{p}{2}}} \left[\frac{4^p}{p^{\frac{4-p}{p}}}+\left(\frac{2p}{4-p}\right)^{\frac{3p-4}{2}}e^{-\frac{4-p}{2}} \left(\frac{2}{p}+\frac{1}{e} \right)\right]
\end{equation*}
\end{lemma}
\begin{proof}
The proof follows applying Lemma \ref{Lrnorm-Green} with $s=p$ and $\sigma=\f{4-p}{2p}<\f{2}{p}$.
\end{proof}

Even if the estimates in Lemma \ref{Gla2} and Lemma \ref{Glap} are not optimal in the exponents of $\la$, they are sufficient to prove a Gagliardo-Nirenberg type inequality and then boundedness from below of the energy functional $F_{\nu}$ when $2<p<4$.

We conclude this subsection by presenting the main properties of $\theta_{\la,\nu}$. 

\begin{lemma}
\label{lem:proptheta}
Let $\nu<0$ and $\la>\nu^{2}$. Then $\theta_{\la,\nu}$ is a strictly increasing function of $\la$ and
\[
 \lim_{\la\to(\nu^{2})^{+}}\theta_{\la,\nu}=-\infty, \quad \lim_{\la\to+\infty} \theta_{\la,\nu}=+\infty.
 \]
\end{lemma}
Moreover, there results that
\begin{equation}
\label{LB-UB-theta}
\f{1}{2\pi}\left(\log\left(\sqrt{\la}+\nu\right)-\f{2\sqrt{\la}}{\sqrt{\la}+\nu}+2\gamma\right)\leq  \theta_{\la,\nu}\leq \f{1}{2\pi}\left(\log\left(\sqrt{\la}+\nu\right)-\f{\sqrt{\la}}{\sqrt{\la}+\nu}+2\gamma\right).
\end{equation}
\begin{proof}
It is straightforward to check that 
\begin{equation*}
\f{d \theta_{\la,\nu}}{d\la}=\f{1}{8\pi\la^{\f{3}{2}}}\left(|\nu|\psi'\left(\f{1}{2}+\f{\nu}{2\sqrt{\la}}\right)+2\sqrt{\la}\right),
\end{equation*}
which is strictly positive since the digamma function $\psi$ is strictly increasing. Moreover, the limit as $\la\to (\nu^{2})^{+}$ follows from the fact that $\psi(x)\to -\infty$ as $x\to 0^{+}$, while the limit as $\la \to+\infty$ is a consequence of the fact that $\theta_{\la,\nu}\sim \f{\log(\la)}{4\pi}$ as $\la\to+\infty$.

Inequalities \eqref{LB-UB-theta} follow applying to the definition of $\theta_{\la,\nu}$ well known inequalities on $\psi(x)$ (see \cite[Eq. (2.2)]{Alzer-97}), i.e.
\begin{equation}
\f{1}{2x}\leq \log(x)-\psi(x)\leq \f{1}{x},\quad \forall\, x>0.
\end{equation}
\end{proof}

\subsection{Gagliardo-Nirenberg type inequalities}
In this subsection, we first recall a well-known Gagliardo-Nirenberg inequality and then prove a modified version of it.

The Gagliardo-Nirenberg inequality in dimension two guarantees the existence of a constant $K_{p}>0$ such that
\begin{equation}
\label{GN-stand}
\|\phi\|_{p}^{p}\leq K_{p}\|\nabla \phi\|_{p}^{p-2}\|\phi\|_{2}^{2}\quad \forall\, \phi \in H^{1}(\Rd).
\end{equation}

The next lemma provides a Gagliardo-Nirenberg type inequality for functions belonging to $\D_{\nu}\setminus H^{1}(\Rd)$ when $\nu<0$.


\begin{lemma}
\label{lem:GNnu<0}
Let $2<p<4$ and $\nu<0$. Then
\begin{equation}
\Vert u \Vert_p^p \leq  2^{p-1}\left[K_{p}\left(1+C_p|q|^{\f{4-p}{2}}\right) \Vert \nabla \phi_\lambda \Vert_2^{p-2}  + \widetilde{C}_{p}|q|^{\f{p}{2}}\right]\|u\|_{2}^{2}\label{GN-mod}
 \end{equation}
 for every $u=\phi_{\la}+q\G_{\la,\nu}$ for which $q\neq0$ and such that 
 \begin{equation*}
 \la\geq \max\left\{1, 4\nu^{2}, \f{|q|^{2}}{\|u\|_{2}^{\f{8}{p}}}\right\}.
 \end{equation*}
 \end{lemma}
\begin{proof}
Let $\nu<0$ and consider $u=\phi_{\la}+q\G_{\la,\nu}$, $q\neq 0$, with $\la\geq\max\{1,4\nu^{2}\}$. Therefore, by \eqref{GN-stand}, Lemma \ref{Gla2}, Lemma \ref{Glap} and the triangle inequality
\[
	\begin{split}
		\Vert u \Vert_p^p  &=   \Vert \phi_\lambda + q \G_{\lambda,\nu} \Vert_p^p \leq 2^{p-1} \left( \Vert \phi_\lambda \Vert_p^p + |q|^p \Vert \G_{\lambda,\nu} \Vert_{{p}}^p\right) \\
		 &\leq  2^{p-1}K_{p} \Vert \nabla \phi_\lambda \Vert_2^{p-2} \Vert \phi_\lambda \Vert_2^2 + 2^{p-1}\widetilde{C}_{p} \frac{|q|^p}{\lambda^{\f{p}{4}}} \\
		&\leq  2^{p-1}K_{p} \Vert \nabla \phi_\lambda \Vert_2^{p-2} \Vert u \Vert_2^2 + 2^{p-1}K_p C_p  \Vert \nabla \phi_\lambda \Vert_2^{p-2} \frac{|q|^2}{\lambda^{\f{p}{4}}} 
		+ 2^{p-1}\widetilde{C}_{p} \frac{|q|^{p}}{\lambda^{\f{p}{4}}}.
		\end{split}
		\]
If in addition $\la\geq \f{|q|^{2}}{\|u\|_{2}^{\frac{8}{p}}}$, then
	\[
	\begin{split}
	\|u\|_{p}^{p}
		&\leq 2^{p-1} K_{p}\left(1+C_p|q|^{\f{4-p}{2}}\right) \Vert \nabla \phi_\lambda \Vert_2^{p-2} \Vert u \Vert_2^2 + 2^{p-1} \widetilde{C}_p |q|^{\f{p}{2}} \Vert u \Vert_2^{2},
	\end{split}
\]
entailing the thesis.
\end{proof}

\begin{remark}
Let us observe that, differently from the standard Gagliardo-Nirenberg inequalities, the inequality in Lemma \ref{lem:GNnu<0} is valid for $2<p<4$ only, since as $p\to 4^{-}$ both $C_{p}$ and $\widetilde{C}_{p}$ of Lemma \ref{Gla2} and Lemma \ref{Glap} diverge. This fact enables us to deal with the $L^{2}$-subcritical case only and not with the $L^{2}$-critical case $p=4$.
\end{remark}

\subsection{Hardy-type inequalities}
In this subsection, we first recall a Hardy-type inequality, proved by Edmunds and Triebel in \cite{Edmunds1999} in a more general context, and then we show a useful application for our purposes.

Before stating the result, given $r>0$ we denote $B_{r}:=\{x\in \R^{2}\,:\, |x|\leq r\}$.
\begin{proposition}
\label{Ed-Trieb}
There exists a constant $c>0$ such that $\forall\, \phi\in H^{1}(\R^{2})$
\begin{equation}
\int_{B_{1}}\f{|\phi(x)|^{2}}{|x|^{2}(1+|\ln|x||)^{2}}\,\ud x\leq c \|\phi\|_{H^{1}(\R^{2})}^{2}\, .
\end{equation}
\end{proposition}

The next result is an application of Proposition \ref{Ed-Trieb}.
\begin{lemma}\label{lem:Quasi-Hardy}
	Let $\phi \in H^1(\mathbb{R}^2)$. Then, for every $0<\ep\leq e^{-1}$, there results
	\begin{equation}
		\int_{\mathbb{R}^2} \frac{|\phi(x)|^2}{|x|}  \, \ud x \ \leq c_{\ep} \Vert \phi \Vert^2_{H^1(\mathbb{R}^2)} + \frac{1}{\varepsilon} \Vert \phi \Vert_{2}^2,
	\end{equation}
	 with $c_{\ep}:=\varepsilon (1+|\ln \varepsilon|)^2$ and $c>0$ given by Proposition \ref{Ed-Trieb}.
\end{lemma}

\begin{proof}
Let $\phi\in H^{1}(\Rd)$. Then, for every $0<\ep\leq 1$
	\[
		\begin{split}
			\int_{\mathbb{R}^2} \frac{|\phi(x)|^2}{|x|}  \, \ud x &= \int_{B_{\varepsilon}} \frac{|\phi(x)|^2}{|x|} \, \ud x + \int_{\mathbb{R}^2 \setminus B_{\varepsilon}} \frac{|\phi(x)|^2}{|x|} \, \ud x \\
			& \leq \int_{B_\varepsilon} \frac{|\phi(x)|^2 |x| (1+|\ln|x||)^2}{|x|^2 (1+|\ln|x||)^2} \, \ud x + \frac{1}{\varepsilon} \int_{\mathbb{R}^2} |\phi(x)|^2 \, \ud x.
		\end{split}
	\]
Note now that the function $\varphi(t):= t (1+|\ln(t)|)^2$ is monotone increasing for $0<t\leq e^{-1}$, whence by using a $L^\infty-L^1$ H\"older inequality, enlarging the domain of integration to $B_{1}$ and applying Proposition \ref{Ed-Trieb}, there results that for every $0<\ep\leq e^{-1}$
	\[
		\begin{split}
			  \int_{\mathbb{R}^2} \frac{1}{|x|} |\phi(x)|^2 \, \ud x &\leq \varepsilon (1+|\ln\varepsilon|)^2 \int_{B_1} \frac{|\phi(x)|^2}{|x|^2(1+\ln|x|)^2} \, \ud x + \frac{1}{\varepsilon} \Vert \phi \Vert_{2}^2 \\
			&\leq \varepsilon(1+|\ln \varepsilon|)^2 c \Vert \phi \Vert_{H^1(\mathbb{R}^2)}^2 + \frac{1}{\varepsilon} \Vert \phi \Vert_{2}^2,
		\end{split}
	\]
	with $c>0$ is the same as in Proposition \ref{Ed-Trieb}.
\end{proof}

\begin{remark}
	Note that $\lim_{\varepsilon \to 0} c_\varepsilon =0$.
\end{remark}

\subsection{Some properties of the symmetric decreasing rearrangement}
\label{subsec-rearr}

We recall here the definition and the properties of the \emph{radially symmetric decreasing rearrangement} of a function in $\R^d$. 

Given $f:\R^d\to[0,+\infty)$ a measurable function \emph{vanishing at infinity}, i.e. $|\{f>t\}|:=|\{x\in\R^d:f(x)>t\}|<+\infty$, for every $t>0$, we call the \emph{radially symmetric decreasing rearrangement} of $f$ the function $f^*:\R^d\to[0,+\infty)$ defined by
\begin{equation}
\label{eq-rearr}
f^*(x)=\int_0^{\infty} \mathds{1}_{\{f>t\}^*}(x)\,\ud t,
\end{equation}
with $\{f>t\}^*$ the open ball centred at zero such that $|\{f>t\}^{*}|=|\{f>t\}|$ and  $\mathds{1}_{\{f>t\}^*}$ the characteristic function of $\{f>t\}^*$.

We recall now some well known properties of $f^{*}$.
First of all, 
\begin{equation}
\label{equimeas}
\|f^*\|_p=\|f\|_p,\qquad\forall f\in L^p(\R^d),\: f\geq0,\quad \forall\, p\geq1.
\end{equation}

Moreover, the well known  \emph{Hardy-Littlewood inequality} holds:  given two measurable functions $f,g:\R^d\to[0,+\infty)$ vanishing at infinity, there results
\begin{equation}
\label{HL-ineq}
\int_{\R^d}f(x)g(x)\,\ud x\le \int_{\R^d}f^*(x)g^*(x)\ud x.
\end{equation} 

By applying inequality \eqref{HL-ineq} (see \cite[Proposition 2.4]{Adami-Boni-Carlone-Tentarelli-2022}), there results that for every $f$, $g\in L^p\left(\R^d; [0,+\infty)\right)$, with $p>1$,
\begin{equation}
\label{ineqf+g}
\int_{\R^d} |f+g|^p \, \ud x\leq \int_{\R^d} |f^*+g^*|^p\,\ud x.
\end{equation} 

Moreover, if $f$ is radially symmetric and strictly decreasing, then the equality in \eqref{HL-ineq} or \eqref{ineqf+g} implies that $g=g^*$ a.e. on $\R^d$.

Finally, the \emph{P\'olya-Szeg\H{o} inequality} holds, i.e. if $f\in H^1(\R^d)$, then $f^*\in H^1(\R^d)$ and in particular
\begin{equation}
\label{PS}
\|\nabla f^*\|_2\le \|\nabla f\|_2.
\end{equation}

\section{Minimizers of the action: proof of Theorem \ref{ex-min-act}}

This section is devoted to the proof of Theorem \ref{ex-min-act}.

We note that, by using \eqref{I-nu-om}, the action functional \eqref{S-nu-om} can be rewritten as
\begin{equation}
\label{S-nu-om-2}
S_{\nu}^{\omega}(v)=\f{p-2}{2p}\|v\|_{p}^{p}+\f{1}{2}I_{\nu}^{\omega}(v)=\f{p-2}{2p}\left(Q_{\nu}(v)+\omega\|v\|_{2}^{2}\right)+\f{1}{p}I_{\nu}^{\omega}(v),
\end{equation}
and thus, \eqref{eq:S_inf} can be rewritten as
\begin{equation}\label{eq:Triangolo}
\f{2p}{p-2}d_{\nu}(\omega)=\inf_{v\in N_{\nu}^{\omega}}\|v\|_{p}^{p}=\inf_{v\in N_{\nu}^{\omega}} \left(Q_{\nu}(v)+\omega\|v\|_{2}^{2}\right) \, .
\end{equation}
\begin{lemma}
\label{lem:dnu<normp}
Let $\nu<0$ and $\omega>\omega_{\nu}$. If $v\in \D_{\nu}$ and $I_{\nu}^{\omega}(v)<0$, then
\begin{equation*}
\|v\|_{p}^{p}>\f{2p}{p-2}d_{\nu}(\omega)\quad\text{and}\quad Q_{\nu}(v)+\omega\|v\|_{2}^{2}>\f{2p}{p-2}d_{\nu}(\omega).
\end{equation*}
\end{lemma}
\begin{proof}
Let $v\in \D_{\nu}$ such that $I_{\nu}^{\omega}(v)<0$. Given $\beta>0$, we have that
\begin{equation*}
I_{\nu}^{\omega}(\beta v)=\beta^{2}\left(Q_{\nu}(v)+\omega\|v\|_{2}^{2}\right)-\beta^{p}\|v\|_{p}^{p}.
\end{equation*}
Since $Q_{\nu}(v)+\omega\|v\|_{2}^{2}\geq (\omega-\omega_{\nu})\|v\|_{2}^{2}>0$, there exists $\beta^{*}\in (0,1)$ such that $I_{\nu}^{\omega}(\beta^{*}v)=0$ and
\begin{equation*}
\f{2p}{p-2}d_{\nu}(\omega)\leq \|\beta^{*}v\|_{p}^{p}=(\beta^{*})^{p}\|v\|_{p}^{p}<\|v\|_{p}^{p}.
\end{equation*}
The other strict inequality can be obtained in a similar way.
\end{proof}
\begin{remark}
\label{equiv-form}
By Lemma \ref{lem:dnu<normp} and \eqref{eq:Triangolo}, it follows that for every $\omega>\omega_{\nu}$
\begin{equation}
\begin{split}
d_{\nu}(\omega)=&\inf\left\{\f{p-2}{2p}\|v\|_{p}^{p}\,:\, v\in \D_{\nu}\setminus\{0\},\, I_{\nu}^{\omega}(v)\leq 0\right\}\\
=&\inf\left\{\f{p-2}{2p}\left(Q_{\nu}(v)+\omega\|v\|_{2}^{2}\right)\,:\, v\in \D_{\nu}\setminus\{0\},\, I_{\nu}^{\omega}(v)\leq 0\right\}.
\end{split}
\end{equation}
\end{remark}

Let us now introduce the energy and the action functional without point interaction. The energy functional is defined as 
\begin{equation}
E_{\nu}(v):=\f{1}{2}\|\na v\|_{2}^{2}+\f{\nu}{2}\left\||x|^{-\f{1}{2}}v\right\|_{2}^{2}-\f{1}{p}\|v\|_{p}^{p} \, .
\end{equation}
It is then convenient to recast our minimization problem as the problem of existence of $u\in H^{1}_{\mu}(\R^{2}):=H^{1}(\R^{2})\cap\{\|u\|_{2}^{2}=\mu\}$ such that
\begin{equation*}
E_{\nu}(u)=\Eps_{\nu}(\mu):=\inf_{v\in H^{1}_{\mu}(\R^{2})}E_{\nu}(v).
\end{equation*}

One can define also the action functional
\begin{equation}
\label{act-nodelta}
\widetilde{S}_{\nu}^{\omega}(v):=E_{\nu}(v)+\f{\omega}{2}\|v\|_{2}^{2}
\end{equation}
and the associated Nehari Manifold
\begin{equation}
\widetilde{N}_{\nu}^{\omega}:=\left\{v\in H^{1}(\Rd)\setminus\{0\}\,:\, \widetilde{I}_{\nu}^{\omega}(v)=0 \right\},
\end{equation}
where
\begin{equation*}
\widetilde{I}_{\nu}^{\omega}(v):=\|\na v\|_{2}^{2}+\nu\left\||x|^{-\f{1}{2}}v\right\|_{2}^{2}+\omega\|v\|_{2}^{2}-\|v\|_{p}^{p}.
\end{equation*}

A minimizer of the action \eqref{act-nodelta} at frequency $\omega$ is a function $u\in \widetilde{N}_{\nu}^{\omega}$ such that
\begin{equation*}
\widetilde{S}_{\nu}^{\omega}(u)=\widetilde{d}_{\nu}(\omega):=\inf_{v\in \widetilde{N}_{\nu}^{\omega}}\widetilde{S}_{\nu}^{\omega}(v)
\end{equation*}

We observe that for every $v\in H^{1}(\Rd)$ it holds $S_{\nu}^{\omega}(v)=\widetilde{S}_{\nu}^{\omega}(v)$ and $N_{\nu}^{\omega}\cap H^{1}(\Rd)=\widetilde{N}_{\nu}^{\omega}$, hence 
\begin{equation}
\label{0<dnu<dnutilde}
0\leq d_{\nu}(\omega)\leq \widetilde{d}_{\nu}(\omega)\quad \forall \omega \in \R.
\end{equation}
\begin{remark}
As observed for $d_{\nu}(\omega)$ in Remark \ref{equiv-form}, it holds that for every $\omega>\nu^{2}$
\begin{equation}
\begin{split}
\widetilde{d}_{\nu}(\omega)&=\inf\left\{\f{p-2}{2p}\|v\|_{p}^{p}\,:\, v\in H^{1}(\Rd)\setminus\{0\},\, \widetilde{I}_{\nu}^{\omega}(v)\leq 0\right\}\\
&=\inf\left\{\f{p-2}{2p}\left(\|\na v\|_{2}^{2}+\nu\left\||x|^{-\f{1}{2}}v\right\|_{2}^{2}+\omega\|v\|_{2}^{2}\right)\,:\, v\in H^{1}(\Rd)\setminus\{0\},\, \widetilde{I}_{\nu}^{\omega}(v)\leq 0\right\}
\end{split}
\end{equation}
\end{remark}

The next proposition collects some results proved in \cite[Proposition 1.2]{Dinh-2021} about the minimizers of the action \eqref{act-nodelta}. 

\begin{proposition}
\label{ex-actmin-attr}
Let $p>2$, $\nu<0$ and $\omega\in\R$.  If $\omega>\nu^{2}$, then $\widetilde{d}_{\nu}(\omega)>0$ and $\widetilde{d}_{\nu}(\omega)$ is attained. Moreover, the minimizers of $\widetilde{d}_{\nu}(\omega)$ are unique, positive and radially symmetric, up to a multiplication by a phase factor.
\end{proposition}

The next lemma shows the mutual relation between $d_{\nu}(\omega)$ and $\widetilde{d}_{\nu}(\omega)$.

\begin{lemma}
\label{dnu<dtildenu}
Let $p>2$, $\nu< 0$ and $\alpha\in\R$. Then
\begin{equation}
\label{eq-levcomp}
d_{\nu}(\omega)<\widetilde{d}_{\nu}(\omega),\quad\forall \omega>\nu^{2}.
\end{equation} 
\end{lemma}
\begin{proof}
Let $\nu<0$, $\omega>\nu^{2}$ and  $g$ be a minimizer of $\widetilde{S}_{\nu}^{\omega}$ at frequency $\omega$. Note that $g$ cannot be a ground state of $S_{\nu}^{\omega}$ at mass $\mu$. Indeed, if this is the case, then $g$ has to satisfy \eqref{bcond}, i.e. $\phi_{\la}(0)=(\alpha+\theta_{\la,\nu})q$. Since $g\in H^{1}(\Rd)$, it follows that $q=0$ and, as a consequence, $g=\phi_{\la}$ and $g(0)=0$. Notice that this contradicts Proposition \ref{ex-actmin-attr}, in particular the fact that $g$ is positive. Since $g$ is not a minimizer of $S_{\nu}^{\omega}$ at mass $\mu$, it follows that there exists $v\in N_{\nu}^{\omega}$ such that $S_{\nu}^{\omega}(v)<S_{\nu}^{\omega}(g)=\widetilde{d}_{\nu}(\omega)$, entailing the thesis.
\end{proof}

\begin{proposition}
\label{ex-part-actmin}
Let $\alpha\in\R$, $\nu<0$ and $\omega>\omega_{\nu}$. Then there exists $u=\phi_{\omega}+q\G_{\omega}\in N_{\nu}^{\omega}$, with $q\neq 0$, such that  $S_{\nu}^{\omega}(u)=d_{\nu}(\omega)$.
\end{proposition}
\begin{proof}
Let $\nu<0$ and $(u_{n})_{n}$ be a minimizing sequence for $S_{\nu}^{\omega}$ in $N_{\nu}^{\omega}$, i.e. $I_{\nu}^{\omega}(u_{n})=0$ for every $n$ and $S_{\nu}^{\omega}(u_{n})\to d_{\nu}(\omega)$ as $n\to+\infty$. Moreover, consider the decomposition corresponding to $\la=\omega$, that is $u_{n}=\phi_{n,\omega}+q_{n}\G_{\omega,\nu}$. We divide the proof into three steps.

\textit{Step 1: weak convergence.}
 By \eqref{S-nu-om-2}, it follows that 
\begin{equation}
\label{conv-d-om}
\f{p-2}{2p}\|u_{n}\|_{p}^{p}\to d_{\nu}(\omega)\,,\quad\f{p-2}{2p}\left(Q_{\nu}(u_{n})+\omega\|u_{n}\|_{2}^{2})\right)\to d_{\nu}(\omega),
\end{equation}
as $n\to+\infty$, hence $(u_{n})$ is bounded in $L^{p}(\Rd)$. Moreover, by Lemma \ref{dnu<dtildenu} and Lemma \ref{lem:Quasi-Hardy} we have that for sufficiently large $n$
\begin{equation*}
\f{p-2}{2p}\left\{\left(1-|\nu|c_{\ep}\right)\|\na\phi_{n,\omega}\|_{2}^{2}+\left(\omega+\f{1}{\ep}\right)\|\phi_{n,\omega}\|_{2}^{2}+|q_{n}|^{2}(\alpha+\theta_{\omega,\nu})\right\}\leq \widetilde{d}_{\nu}(\omega),
\end{equation*} 
so that, choosing $0<\ep\leq e^{-1}$ satisfying $c_{\ep}<|\nu|^{-1}$, it follows that $(\phi_{n,\omega})_{n}$ is bounded in $H^{1}(\Rd)$ and $(q_{n})_{n}$ is bounded in $\C$. Therefore, there exists $u\in \D_{\nu}$, $\phi_{\omega}\in H^{1}(\Rd)$ and $q\in \C$ such that, up to subsequences, $u_{n}\deb u$ weakly in $L^{p}(\Rd)$, $\phi_{n,\omega}\deb \phi_{\omega}$ weakly in $H^{1}(\Rd)$, $q_{n}\to q$ in $\C$ and $u=\phi_{\omega}+q\G_{\omega}$

\textit{Step 2: $q\neq 0$.} If we suppose by contradiction that $q=0$, then by \eqref{conv-d-om} there results
\begin{equation*}
\f{p-2}{2p}\|\phi_{n,\omega}\|_{p}^{p}\to d_{\nu}(\omega)\,,
\end{equation*}
as $n\to+\infty$. Let us denote with
\begin{equation*}
\beta_{n}:=\left(\f{\|\na\phi_{n,\omega}\|+\omega\|\phi_{n,\omega}\|_{2}^{2}+\nu\left\|\phi_{n,\omega}|x|^{-\f{1}{2}}\right\|_{2}^{2}}{\|\phi_{n,\omega}\|_{p}^{p}}\right)^{\f{1}{p}}.
\end{equation*}
We observe that $\beta_{n}\to 1$ and $\widetilde{I}_{\nu}^{\omega}(\beta_{n}\phi_{n,\omega})=0$, thus by Lemma \ref{dnu<dtildenu} we have
\begin{equation*}
\widetilde{d}_{\nu}(\omega)\leq \f{p-2}{2p}\|\beta_{n}\phi_{n,\omega}\|_{p}^{p}\to d_{\nu}(\omega)<\widetilde{d}_{\nu}(\omega),
\end{equation*}
getting a contradiction, hence $q\neq 0$.

\textit{Step 3: strong convergence and conclusion.} There results that 
\begin{equation}
\label{BL2}
\|u_{n}-u\|_{p}^{p}=\|u_{n}\|_{p}^{p}-\|u\|_{p}^{p}+o(1)
\end{equation}
and 
\begin{equation}
\label{I-un-u}
I_{\nu}^{\omega}(u_{n}-u)=I_{\nu}^{\omega}(u_{n})-I_{\nu}^{\omega}(u)+o(1),
\end{equation}
as $n\to +\infty$. Identity \eqref{BL2} is a consequence of Brezis-Lieb lemma \cite{Brezis-Lieb-83}, since $(u_{n})_{n}$ is bounded in $L^{p}(\Rd)$ and $u_{n}\to u$ a.e. in $\Rd$. \eqref{I-un-u} follows instead by the fact that $\phi_{n,\omega}\deb \phi_{\omega}$ weakly in $H^{1}(\Rd)$, $q_{n}\to q$ in $\C$ and 
\begin{equation}
\label{phin-phi-Coul}
\int_{\R^{2}}|x|^{-1}(\phi_{n,\la}-\phi_{\la})\overline{\phi_{\la}}\,\ud x\to 0,\quad n \to+\infty.
\end{equation}

{
To prove \eqref{phin-phi-Coul}, without loss of generality we can pick $v \in C^\infty_c(\mathbb{R}^2)$, call $r_n=\phi_{n,\lambda}-\phi_\lambda$ and, after fixing $\ep>0$, we prove that $\int_{\R^{2}}|x|^{-1}r_{n}(x)v(x)\,dx\leq \ep$ for $n$ large enough. In particular,
\begin{equation}
	\begin{split}
		\left| \int_{\mathbb{R}^2} |x|^{-1} r_n v(x) \, \ud x \right| &\leq  \left| \int_{B_\eta} |x|^{-1} r_n(x) v(x) \, \ud x \right| + \left|\int_{\mathbb{R}^2\setminus B_\eta} |x|^{-1} r_n(x) v(x) \, \ud x \right| \\
		&= J_1+J_2 \, .
	\end{split}
\end{equation}

By using Cauchy-Schwarz inequality and Lemma \ref{lem:Quasi-Hardy} (with $\ep=e^{-1}$), we get
\begin{equation*}
	\begin{split}
		J_1 &\leq \left(\int_{B_\eta} |x|^{-1} |v(x)|^2 \, \ud x \right)^{\frac{1}{2}} \left(\int_{B_\eta} |x|^{-1} |r_n(x)|^2 \, \ud x \right)^{\frac{1}{2}} \\
		&\leq \sqrt{\left(\f{4}{e}+e\right)c}\Vert r_n \Vert_{H^1(\mathbb{R}^2)}\left(\int_{B_\eta} |x|^{-1} |v(x)|^2 \, \ud x \right)^{\frac{1}{2}}  \, ,
	\end{split}
\end{equation*}
thus, since $(u_{n})$ is bounded in $H^{1}(\Rd)$ and $v\in H^{1}(\Rd)$, for $\eta$ small enough $J_{1}\leq \f{\ep}{2}$. Now, fix $\eta$ such that $J_{1}\leq\f{\ep}{2}$. Then
\begin{equation*}
	J_2 \leq \frac{1}{\eta} \int_{\mathbb{R}^2 \setminus B_\eta} |v(x)| |r_n(x)| \, \ud x \leq \frac{1}{\eta} \Vert v \Vert_\infty \int_{\mathrm{supp}(v)} |r_n| \, \ud x \,,
\end{equation*}
hence for $n$ large enough $J_{2}\leq \f{\ep}{2}$, since $r_n \deb 0$ in $H^1(\mathbb{R}^2)$ and, for the compact embedding, $r_{n}\to 0$ in $L^1_{loc}(\mathbb{R}^2)$. }

Now, using \eqref{BL2} and the fact that $\|u\|_{p}^{p}>0$,
\begin{equation*}
\lim_{n\to+\infty} \f{p-2}{2p}\|u_{n}-u\|_{p}^{p}<\lim_{n\to+\infty}\f{p-2}{2p}\|u_{n}\|_{p}^{p}=d_{\nu}(\omega),
\end{equation*}
hence $I_{\nu}^{\omega}(u_{n}-u)\geq 0$ for sufficiently large $n$ by Lemma \ref{lem:dnu<normp} and  $I_{\nu}^{\omega}(u)\leq 0$ by \eqref{I-un-u}. Thus, from the lower semicontinuity of the norms and Remark \ref{equiv-form} 
\begin{equation*}
d_{\nu}(\omega)\leq\f{p-2}{2p}\|u\|_{p}^{p}\leq\f{p-2}{2p}\liminf_{n\to+\infty}\|u_{n}\|_{p}^{p}=d_{\nu}(\omega),
\end{equation*} 
hence $u_{n}\to u$ in $L^{p}(\Rd)$ and $\f{p-2}{2p}\|u\|_{p}^{p}=d_{\nu}(\omega)$, so that the thesis follows via Lemma \ref{lem:dnu<normp}.
\end{proof}

\begin{lemma}
\label{lem:dnu-bounds}
Let $\alpha\in \R$, $\nu<0$ and $\omega>\omega_{\nu}$. If $u\in \D_\nu$ satisfies
\begin{equation}
\label{dnu-bounds}
\f{p-2}{2p}\left(Q_{\nu}(u)+\omega\|u\|_{2}^{2}\right)\leq d_{\nu}(\omega)\leq \f{p-2}{2p}\|u\|_{p}^{p},
\end{equation}
then $u\in N_{\nu}^{\omega}$ and $S_{\nu}^{\omega}(u)=d_{\nu}(\omega)$.
\end{lemma}
\begin{proof}
From \eqref{dnu-bounds}, we have $I_{\nu}^{\omega}(u)\leq 0$ and using also \eqref{S-nu-om-2}, we get  $S_{\nu}^{\omega}(u)\leq d_{\nu}(\omega)$. Moreover, again from \eqref{dnu-bounds} and Lemma \ref{lem:dnu<normp} we get $I_{\nu}^{\omega}(u)\geq 0$, thus $u\in N_{\nu}^{\omega}$ and $S_{\nu}^{\omega}(u)=d_{\nu}(\omega)$.
\end{proof}

The next three lemmas characterize the minimizers of the action \eqref{S-nu-om} at fixed frequency $\omega$.

\begin{lemma}
\label{q-phi-neq0}
Let $\alpha\in \R$, $\nu<0$, $\omega>\omega_{\nu}$ and $u$ be a minimizer of the action $S_{\nu}^{\omega}$ at frequency $\omega$. Then $q\neq 0$ and $\phi_{\la}:=u-q\G_{\la,\nu}\neq 0$ for every $\la>\nu^{2}$.
\end{lemma}
\begin{proof}
Let $\la>\nu^{2}$ and consider the decomposition $u=\phi_{\la}+q\G_{\la}$. Assume by contradiction that $\phi_{\la}=0$. Since $u\neq 0$, we have $q\neq 0$, and, since $u$ satisfies \eqref{bcond}, then $\alpha+\theta_{\la,\nu}=0$, which, for $\lambda > \nu^2$ is satisfied if and only if $\la=\omega_{\nu}$. Moreover, since $u$ satisfies also \eqref{EL-eq}, it follows that
\begin{equation*}
\omega-\omega_{\nu}=|q|^{p-2}\left|\G_{\omega_{\nu},\nu}\right|^{p-2}\quad \forall\,x\in \Rd\setminus\{0\},
\end{equation*}
which is a contradiction, hence $\phi_{\la}\neq 0$. If instead we suppose by contradiction that $q=0$, i.e $u=\phi_{\la}\in H^{1}(\Rd)$, then $d_{\nu}(\omega)=S_{\nu}^{\omega}(u)=\widetilde{d}_{\nu}(\omega)$, contradicting Lemma \ref{dnu<dtildenu}.
\end{proof}

\begin{lemma}
\label{u-pos}
Let $\alpha\in \R$, $\nu<0$, $\omega>\omega_{\nu}$ and $u$ be a minimizer of the action $S_{\nu}^{\omega}$ at frequency $\omega$. If $q$ is a positive real number, then $u$ is positive.
\end{lemma}
\begin{proof}
Let $u=\phi_{\omega}+q\G_{\omega,\nu}$ be a minimizer of the action $S_{\nu}^{\omega}$ at frequency $\omega$ and $\Omega:=\{x\in\Rd\,:\, \phi_{\omega}(x)\neq 0\}$. By Lemma \ref{q-phi-neq0}, $|\Omega|>0$, so that $\phi_{\omega}(x)=e^{i\theta(x)}|\phi_{\omega}(x)|$ for every $x\in \Rd\setminus\{0\}$, for some $\theta:\Omega\to[0,2\pi)$. If we prove that $\theta(x)=0$ for a.e. $x\in \Omega\setminus \{0\}$, then the proof is complete since in this case $\phi_{\omega}(x)=|\phi_{\omega}(x)|\geq 0$ for every $x\in \Rd$, whence $u(x)>0$ for every $x\in \Rd\setminus\{0\}$. Suppose by contradiction that $\theta\neq 0$ on $\Omega_{1}\subset\left(\Omega\setminus\{0\}\right)$, with $|\Omega_{1}|>0$, and consider the function $\widehat{u}:=|\phi_{\omega}|+q\G_{\omega,\nu}$. On the one hand, we have that $|u(x)|^{2}<|\widehat{u}(x)|^{2}$ for every $x\in \Omega_{1}$, thus $d_{\nu}(\omega)=\f{p-2}{2p}\|u\|_{p}^{p}<\f{p-2}{2p}\|\widehat{u}\|_{p}^{p}$ since $|\Omega_{1}|>0$. Moreover, using \eqref{equimeas} and \eqref{PS}, we have $\f{p-2}{2p}\left(Q_{\nu}(\widehat{u})+\omega\|\widehat{u}\|_{2}^{2}\right)<\f{p-2}{2p}\left(Q_{\nu}(u)+\omega\|u\|_{2}^{2}\right)=d_{\nu}(\omega)$, hence
\begin{equation*}
\f{p-2}{2p}\left(Q_{\nu}(\widehat{u})+\omega\|\widehat{u}\|_{2}^{2}\right)<d_{\nu}(\omega)<\f{p-2}{2p}\|\widehat{u}\|_{p}^{p}.
\end{equation*}
From Lemma \ref{lem:dnu-bounds}, we deduce that $\widehat{u}\in N_{\nu}^{\omega}$ and $S_{\nu}^{\omega}(\widehat{u})=\f{p-2}{2p}\|\widehat{u}\|_{p}^{p}=d_{\nu}(\omega)=\f{p-2}{2p}\|u\|_{p}^{p}$, getting a contradiction.
\end{proof}
\begin{lemma}
\label{u-sym}
Let $\alpha\in \R$, $\nu<0$, $\omega>\omega_{\nu}$ and $u$ be a minimizer of the action $S_{\nu}^{\omega}$ at frequency $\omega$. If $u$ is positive, then $u$ is radially symmetric and decreasing along the radial direction.
\end{lemma}
\begin{proof}
Let $u=\phi_{\omega}+q\G_{\omega,\nu}$ be a minimizer of the action $S_{\nu}^{\omega}$ at frequency $\omega$. Since $\G_{\omega,\nu}$ is radially symmetric and decreasing, it is sufficient to prove that $\phi_{\omega}=\phi_{\omega}^{*}$, with $\phi_{\omega}^{*}$ the radially symmetric decreasing rearrangement of $\phi_{\omega}$. Assume by contradiction that $\phi_{\omega}\neq \phi_{\omega}^{*}$ and define the function $\widetilde{u}:=\phi_{\omega}^{*}+q\G_{\omega,\nu}$. By \eqref{equimeas}, \eqref{PS} and the equality cases in \eqref{HL-ineq} and \eqref{ineqf+g}, there results that
\begin{equation*}
Q_{\nu}(\widetilde{u})+\omega\|\widetilde{u}\|_{2}=\|\na\phi_{\omega}^{*}\|_{2}^{2}+\omega\|\phi_{\omega}^{*}\|_{2}^{2}+\nu\left\||x|^{-\f{1}{2}}\phi_{\omega}^{*}\right\|_{2}^{2}+|q|^{2}\left(\alpha+\theta_{\omega,\nu}\right)<Q_{\nu}(u)+\omega\|u\|_{2}^{2}, 
\end{equation*}
and
\begin{equation*}
I_{\nu}^{\omega}(\widetilde{u})=Q_{\nu}(\widetilde{u})+\omega\|\widetilde{u}\|_{2}-\|\widetilde{u}\|_{p}^{p}<I_{\nu}^{\omega}(u),
\end{equation*}
thus $I_{\nu}^{\omega}(\widetilde{u})<0$ and $\f{p-2}{2p}\left(Q_{\nu}(\widetilde{u})+\omega\|\widetilde{u}\|_{2}\right)<d_{\nu}(\omega)$, which is in contradiction with Lemma \ref{lem:dnu<normp}. 
\end{proof}

We are now ready to prove Theorem \ref{ex-min-act}.

\begin{proof}[Proof of Theorem \ref{ex-min-act}]
The existence part follows from Lemma  \ref{ex-part-actmin}. Let now $\omega>\omega_{\nu}$ and $u$ be a minimizer of the action $S_{\nu}^{\omega}$ at frequency $\omega$. By Lemma \ref{q-phi-neq0}, we know that $q\neq 0$, hence without loss of generality we can assume that $q>0$ by multiplying by a proper phase factor. By applying Lemma \ref{u-pos} and Proposition \ref{u-sym}, we conclude the proof.
\end{proof}

\section{Fixed mass ground states: proof of Theorem \ref{ex-gs-mu}}

This section is devoted to the proof of Theorem \ref{ex-gs-mu}. We start by recalling some results about the energy $E_{\nu}$, that were partially proved in \cite{Dinh-2021}.

\begin{proposition}
\label{ex-Enu}
Let $p\in (2,4)$, $\mu>0$ and $\nu<0$. Then 
there exists $u\in H^{1}_{\mu}(\R^{2})$ such that 
\begin{equation*}
E_{\nu}(u)=\Eps_{\nu}(\mu)<\Eps_{0}(\mu).
\end{equation*}
Moreover, every ground state is positive, radially symmetric and decreasing along the radial direction, up to a multiplication by a constant phase factor.
\end{proposition}
\begin{proof}
Using \cite[Theorem 1.8]{Dinh-2021} we obtain the existence part and the fact that every ground state is nonnegative and radially symmetric up to a multiplication by a phase factor, thus it remains to prove that every ground state is decreasing along the radial direction up to a phase factor. Suppose by contradiction that there exists a positive ground state $u\in H^{1}(\Rd)$ such that $u$ is not decreasing along the radial direction. If we consider the radially symmetric decreasing rearrangement $u^{*}$ of $u$, then by applying \eqref{HL-ineq} with $f(x)=\f{1}{|x|}$ and $g=u$ we get that 
\begin{equation}
\label{HL-appl}
\int_{\Rd}\f{|u^{*}(x)|^{2}}{|x|}\,\ud x > \int_{\Rd}\f{|u(x)|^{2}}{|x|}\,\ud x,
\end{equation}
where the strict inequality follows from the fact that $u$ is not decreasing along the radial direction.
By \eqref{PS}, \eqref{equimeas} and \eqref{HL-appl} and using the fact that $\nu<0$, there results that 
 \begin{equation}
E_{\nu}(u^{*})<E_{\nu}(u),
\end{equation}  
being a contradiction and concluding the proof.
\end{proof}

\begin{proposition}
\label{Eps<Eps0}
Let $p\in(2,4)$, $\nu< 0$ and $\alpha\in\R$. Then, 
\begin{equation}
\label{eq-levcomp2}
\F_{\nu}(\mu)<\Eps_{\nu}(\mu)<0,\quad\forall \mu>0.
\end{equation} 
\end{proposition}
\begin{proof}
Let $\nu<0$, $\mu>0$ and  $g$ be a ground state of $E_{\nu}$ of mass $\mu$. Note that $g$ cannot be a ground state of $F_{\nu}$ at mass $\mu$. Indeed, if this is the case, then $g$ has to satisfy \eqref{bcond}, that is $\phi_{\la}(0)=(\alpha+\theta_{\la,\nu})q$. Since $g\in H^{1}(\Rd)$, it follows that $q=0$ and, as consequence, that $g=\phi_{\la}$ and $g(0)=0$, but this contradicts $(iii)$ of Proposition \ref{ex-Enu}, in particular the fact that $g$ is positive. Since $g$ is not a ground state of $F_{\nu}$ at mass $\mu$, it follows that there exists $v\in \D_{\nu}^{\mu}$ such that $F_{\nu}(v)<F_{\nu}(g)=\Eps_{\nu}(\mu)$, entailing the thesis.
\end{proof}

\begin{lemma}
\label{q>C}
Let $p\in (2,4)$, $\alpha\in \R$, $\nu<0$ and $\mu>0$. If $u_{n}=\phi_{n,\la}+q_{n}\G_{\la,\nu}$ is a minimizing sequence of mass $\mu$ for the energy $F_{\nu}$, then there exists $\overline{n}\in \mathbb{N}$ and a constant $C$ such that $|q_{n}|>C$ for every $n\geq \overline{n}$.
\end{lemma}
\begin{proof}
We prove it by contradiction. Suppose that there exists a subsequence of $q_{n}$, that we do not rename, such that $q_{n}\to 0$. Then, $\|\phi_{n,\la}\|_{2}^{2}-\mu\to 0$ as $n\to+\infty$, so that $\|\phi_{n,\la}\|_{2}$ is bounded. Moreover, by applying \eqref{GN-stand} to \eqref{en-fan}, there results
\begin{equation*}
\begin{split}
F_{\nu}(u_{n}) &\geq \f{1}{2}\|\nabla \phi_{n,\la}\|_{2}^{2}+\f{\la}{2}\left(\|\phi_{n,\la}\|_{2}^{2}-\mu\right)+\f{\nu}{2}\left\||x|^{-\f{1}{2}}\phi_{n,\la}\right\|_{2}^{2}+\f{\alpha+\theta_{\la,\nu}}{2}|q_{n}|^{2}\\
&-\f{2^{p-1}}{p}\left(K_{p}\|\nabla \phi_{n,\la}\|_{2}^{p-2}\|\phi_{n,\la}\|_{2}^{2}+|q_{n}|^{p}\|\G_{\la,\nu}\|_{p}^{p}\right)\\
&=\f{1}{2}\|\nabla \phi_{n,\la}\|_{2}^{2}-\f{2^{p-1}K_{p}}{p}\|\nabla \phi_{n,\la}\|_{2}^{p-2}\|\phi_{n,\la}\|_{2}^{2}+\f{\nu}{2}\left\||x|^{-\f{1}{2}}\phi_{n,\la}\right\|_{2}^{2}+o(1)
\end{split}
\end{equation*}
as $n\to+\infty$. By applying Lemma \ref{lem:Quasi-Hardy} with $0<\ep\leq e^{-1}$ such that $c_{\ep}=\min\left\{\f{1}{2|\nu|},\f{4c}{e}\right\}$, where $c>0$ is the same as in Proposition \ref{Ed-Trieb}, we get
\begin{equation*}
F_{\nu}(u_{n}) \geq \f{1}{4}\|\nabla \phi_{n,\la}\|_{2}^{2}-\f{C_{p}}{p}\|\nabla \phi_{n,\la}\|_{2}^{p-2}\|\phi_{n,\la}\|_{2}^{2}-\f{|\nu|}{2}\left(c_{\ep}+\f{1}{\ep}\right)\|\phi_{n,\la}\|_{2}^{2}+o(1),
\end{equation*} 
so that $\|\nabla \phi_{n,\la}\|_{2}$ is bounded since $F_{\nu}(u_{n})$ is bounded from above and $p<4$.

We introduce the sequence $\xi_{n}:=\f{\sqrt{\mu}}{\|\phi_{n,\la}\|_{2}}\phi_{n,\la}$, for which $\|\xi_{n}\|_{2}^{2}=\mu$ and $\|\nabla \xi_{n}\|_{2}^{2}$ is bounded. Therefore, using the fact that $q_{n}\to 0$ and $\phi_{n,\la}-u_{n}\to 0$ strongly in every $L^{p}(\R^{2})$ with $2\leq p<\infty$, one gets
\begin{equation*}
F_{\nu}(u_{n})=E_{\nu}(\phi_{n,\la})+o(1)=E_{\nu}(\xi_{n})+o(1)\geq \Eps_{\nu}(\mu)+o(1),\quad n\to +\infty.
\end{equation*}
Passing to the limit, there results $\F_{\nu}(\mu)\geq \Eps_{\nu}(\mu)$, being in contradiction with Proposition \ref{Eps<Eps0}.
\end{proof}

\begin{lemma}
\label{minimseq2}
Let $p\in(2,4)$, $\alpha\in\R$, $\nu<0$ and $\mu>0$. If  $(u_{n})_n$ is a minimizing sequence of mass $\mu$ for the energy $F_{\nu}$, then it is bounded in $L^r(\Rd)$ for every $r\geq2$, and there exists $u\in \D_{\nu}\setminus H^1(\Rd)$ such that, up to subsequences,
\begin{equation}
\label{lim-1}
u_n\deb u \quad \text{in}\quad L^2(\Rd), \quad u_n\to u \quad\text{a.e. in} \quad \Rd, \quad n \to +\infty.
\end{equation}
In particular, if one fixes $\lambda\geq \max\{1,4\nu^{2}\}$ and the decomposition $u_{n}=\phi_{n,\la}+q_{n}\G_{\la,\nu}$, then $(\phi_{n,\la})_n$ and $(q_n)_n$ are bounded in $H^1(\Rd)$ and $\C$, respectively, and there exist $\phi_{\la}\in H^{1}(\Rd)$ and $q\in \C\setminus\{0\}$ such that $u=\phi_\la+q\G_{\la,\nu}$ and, up to subsequences,
\begin{equation}
\label{lim-2}
\phi_{n,\la}\deb\phi_{\la} \quad \text{in}\quad L^2(\Rd), \quad \na\phi_{n,\la}\deb\na\phi_{\la} \quad \text{in}\quad L^2(\Rd),\quad q_{n}\rightarrow q\quad \text{in}\quad \C, \quad n \to +\infty.
\end{equation}
\end{lemma}
\begin{proof}
Let $(u_{n})_n$ be a  minimizing sequence of mass $\mu$ for the energy $F_{\nu}$. By Banach-Alaoglu Theorem, there exists $u\in L^{2}(\Rd)$ such that $u_n\deb u$ in $L^2(\Rd)$ up to subsequences. Moreover, owing to Lemma \ref{q>C} we can suppose without loss of generality that there exists $C>0$ such that $|q_n| > C$ for every $n$, then we can rely on the decomposition $u_{n}=\phi_{n}+q_{n}\G_{\la_{n},\nu}$ with $\la_{n}:=M\f{|q_{n}|^{2}}{\|u_{n}\|_{2}^{\f{8}{p}}}$ and $M=\max\left\{1,\f{\mu^{\f{4}{p}}}{C^{2}}, \f{4\nu^{2}\mu^{\f{4}{p}}}{C^{2}}\right\}$. By this choice of $M$, there results that 
 \begin{equation*}
 \la_{n}\geq \max\left\{1, 4\nu^{2}, \f{|q_{n}|^{2}}{\|u_{n}\|_{2}^{\f{8}{p}}}\right\},
 \end{equation*}
hence Lemma \ref{Gla2}, Lemma \ref{Glap} and Lemma \ref{lem:GNnu<0} apply. In view of this decomposition, by the triangle inequality and Lemma \ref{Gla2} there results
\begin{equation}
\label{phin-mu}
\|\phi_{n}\|_{2}^{2}\leq 2\left(\|u_{n}\|_{2}^{2}+|q_{n}|^{2}\|\G_{\la_{n},\nu}\|_{2}^{2}\right)=2\left(1+C_{p}|q_{n}|^{\f{4-p}{2}}\right)\mu,
\end{equation}
and the energy reads as 
\begin{multline}
\label{Fanu-q0}
F_{\nu}(u_{n})=\f{\|\na\phi_{n}\|^2_{2}}{2}+\f{M}{2}\f{|q_{n}|^{2}}{\|u_{n}\|_{2}^{\f{8}{p}}}\|\phi_{n}\|^2_{2}+\f{\nu}{2}\left\||x|^{-\f{1}{2}}\phi_{n}\right\|_{2}^{2}\\
+\left(\alpha-\f{M}{\|u_{n}\|_{2}^{\f{8}{p}-2}}+\theta_{\la_{n},\nu}\right)\f{|q_{n}|^{2}}{2}-\f{\|u_{n}\|^{p}_{p}}{p}.
\end{multline}

By applying Lemma \ref{lem:GNnu<0}, Lemma \ref{lem:Quasi-Hardy} (with $0<\ep\leq e^{-1}$ to be chosen), \eqref{LB-UB-theta} and \eqref{phin-mu} to \eqref{Fanu-q0} and by recalling that $\la_{n}\geq 4\nu^{2}$, there results
\begin{equation}
\begin{split}
F_{\nu}(u_{n})&\geq \f{1-|\nu|c_{\ep}}{2}\|\na\phi_{n}\|^2_{2}+\left(2\pi\left(\alpha-M\mu^{1-\f{4}{p}}\right)+\log\left(\f{M|q_{n}|}{\mu^{\f{2}{p}}}+\nu\right)-4+2\gamma\right)\f{|q_{n}|^2}{4\pi}\\
-&|\nu|\left(c_{\ep}+\ep^{-1}\right)\left(1+C_{p}|q_{n}|^{\f{4-p}{2}}\right)\mu-2^{p-1}\left[K_{p}\left(1+C_p|q_{n}|^{\f{4-p}{2}}\right) \Vert \nabla \phi_n \Vert_2^{p-2}  + \widetilde{C}_{p}|q_{n}|^{\f{p}{2}}\right]\mu.
\end{split}
\end{equation}

Since $c_{\ep}\to 0$ as $\ep\to 0$, it is possible to choose $0<\ep\leq e^{-1}$ such that $c_{\ep}\leq\f{1}{2|\nu|}$. Consider now the function $g:\R^{+}\times\R^{+}\to \R$ defined as
\begin{equation*}
g(x,y):=\f{x^{2}}{4}+\f{c_{1}+\log\left(M\mu^{-\f{2}{p}}y\right)}{4\pi}y^{2}-\left(1+C_{p}y^{\f{4-p}{2}}\right)\left(c_{2}+c_{3}x^{p-2}\right)-c_{4}y^{p-2},
\end{equation*}
with $c_{1}:=2\pi\left(\alpha-M\mu^{1-\f{4}{p}}\right)-4+2\gamma$, $c_{2}:=|\nu|\left(c_{\ep}+\ep^{-1}\right)\mu$, $c_{3}:=2^{p-1}K_{p}\mu$ and $c_{4}:=2^{p-1}\widetilde{C_{p}}\mu$.
Since $\f{4-p}{2}<2$, $p-2<2$, $\frac{p}{2}<2$ and $\f{4-p}{2}+p-2<2$ if $2<p<4$, there results that 
\[
g(x,y)\to +\infty\quad \text{as}\quad x^{2}+y^{2}\to+\infty.
\]
Combing the facts that $F_{\nu}(u_{n})\geq g\left(\|\na\phi_{n}\|_{2}, |q_{n}|\right)\to +\infty$ as $\|\na\phi_{n}\|_{2}^{2}+|q_{n}|^{2}\to+\infty$ and $F_{\nu}(u_{n})\to \F_{\nu}(\mu)<0$ as $n\to+\infty$, it follows that $(\nabla\phi_n)_n$ is bounded in $L^2(\Rd)$ and $(q_n)_n$ is bounded in $\C$, so that, up to subsequences, $q_n\to q$ and $q \neq 0$ since $|q_n| > C > 0$. Furthermore, in view of \eqref{phin-mu} also $(\phi_{n})_{n}$ is bounded in $L^{2}(\Rd)$.

Fix now $\la \geq 4\nu^{2}$ and consider the decomposition $u_{n}=\phi_{n,\la}+q_{n}\G_{\la,\nu}$ with $\phi_{n,\la}:=\phi_{n}+q_{n}(\G_{\la_{n},\nu}-\G_{\la,\nu})$. By \eqref{phin-mu} and Lemma \ref{Gla2}, we have that
\begin{equation*}
\|\phi_{n,\la}\|_{2}^{2}\leq 2\|\phi_{n}\|_{2}^{2}+4|q_{n}|^{2}\left(\|\G_{\la_{n},\nu}\|_{2}^{2}+\|\G_{\la,\nu}\|_{2}^{2}\right)\leq 4\left(1+2C_{p}|q_{n}|^{\f{4-p}{2}}\right)\mu+\f{4C_{p}|q_{n}|^{2}}{\la^{\f{p}{4}}},
\end{equation*}
hence $(\phi_{n,\la})_{n}$ is bounded in $L^{2}(\Rd)$. Moreover, by using the convexity of $f(t)=t^{p}$ together with \eqref{GN-stand}, Lemma \ref{lem:GNnu<0} and Lemma \ref{lem:Quasi-Hardy} (with $\ep$ to be chosen), we get
\begin{equation*}
\begin{split}
F_{\nu}(u_{n})\geq &\f{1-|\nu|c_{\ep}}{2}\|\nabla \phi_{n,\la}\|_{2}^{2}-\f{2^{p-1}K_{p}}{p}\|\na \phi_{n,\la}\|_{2}^{p-2}\|\phi_{n,\la}\|_{2}^{2}-\f{\la}{2}\mu\\
&-\f{|\nu|}{2}\left(c_{\ep}+\f{1}{\ep}\right)\|\phi_{n,\la}\|_{2}^{2}+\f{\alpha+\theta_{\la,\nu}}{2}|q_{n}|^{2}-\f{2^{p-1}\widetilde{C}_{p}}{p\la^{\f{p}{4}}}|q_{n}|^{p}.
\end{split}
\end{equation*}
By choosing $0<\ep\leq e^{-1}$ such that $c_{\ep}=\f{1}{2|\nu|}$ as above and using the boundedness of $(q_{n})_{n}$ and $\|\phi_{n,\la}\|_{2}$, we deduce that $(\na\phi_{n,\la})_{n}$ is bounded in $L^{2}(\Rd)$, hence there exists $\phi_{\la}\in H^{1}(\Rd)$ such that, up to subsequences, $\phi_{n,\la}\deb \phi$ and $\na \phi_{n,\la}\deb \na \phi_{\la}$ in $L^{2}(\Rd)$: in particular, $\phi_\la=u-q\G_\la$.
Furthermore, by Rellich-Kondrakov theorem, $\phi_{n,\la}\to\phi_\la$ in $L^r_{loc}(\Rd)$, for every $r>2$, so that $u_n\to u$ a.e. in $\Rd$.
\end{proof}

The next lemma establishes a connection between ground states at fixed mass of \eqref{en-fan} and minimizers of the action \eqref{S-nu-om} at fixed frequency, allowing to deduce the properties of fixed mass ground states by knowing the properties of action minimizers. Such a lemma has been proved in \cite{Dov-Serra-Tilli-22} and \cite{Jeanjean-22} for the standard NLS energy: we report here the proof for the sake of completeness.

\begin{lemma}
\label{GSareAM}
Let $p>2$, $\alpha\in \R$, $\nu\in \R\setminus\{0\}$ and $\mu>0$. If $u$ is a ground state of $F_{\nu}$ of mass $\mu$, then it is also a minimizer of the action $S_{\nu}^{\omega}$ at frequency $\omega>\mu^{-1}(\|u\|_{p}^{p}-Q_{\nu}(u))$. Moreover, if $\nu<0$, then $\omega>\omega_{\nu}$.
\end{lemma}
\begin{proof}
Let $u$ be a ground state of $F_{\nu}$ of mass $\mu$ and let $\omega>0$ be the associated Lagrange multiplier, given by $\omega=\mu^{-1}(\|u\|_p^p-Q_{\nu}(u))$. Suppose, by contradiction, that there exists $v=\eta_\la+\xi\G_{\la,\nu} \in N_{\nu}^{\omega}$ such that $S_{\nu}^{\omega}(v)<S_{\nu}^{\omega}(u)$ and let $\beta>0$ be such that $\|\beta v\|_2^2=\mu$. Then
\begin{equation*}
S_{\nu}^{\omega}(\beta v)=\f{\beta^2}{2}(Q_{\nu}(v)+\omega\|v\|_{2}^{2})-\f{\beta^p}{p}\|v\|_p^p.
\end{equation*}
Computing the derivative with respect to $\beta$ and using that $v\in N_{\nu}^{\omega}$, we get 
\begin{equation*}
\f{d}{d\beta}S_{\nu}^{\omega}(\beta v)=\beta (Q_{\nu}(v)+\omega\|v\|_{2}^{2})-\beta^{p-1}\|v\|_p^p=\beta I_{\nu}^{\omega}(v)+(\beta-\beta^{p-1})\|v\|_p^p=\beta(1-\beta^{p-2})\|v\|_p^p,
\end{equation*}
which is greater than zero if and only if $0<\beta< 1$. Hence $S_{\nu}^{\omega}(\beta v)\le S_{\nu}^{\omega}(v)$, for every $\beta>0$. Therefore, since $S_{\nu}^{\omega}(\beta v)\le S_{\nu}^{\omega}(v) <S_{\nu}^{\omega}(u)$,
\begin{equation*}
F_{\nu}(\beta v)+\f{\omega}{2}\|\beta v\|_2^2<F_{\nu}(u)+\f{\omega}{2}\|u\|_2^2,
\end{equation*}
and using the fact that $\|\beta v\|_2^2=\|u\|_2^2=\mu$, this entails $F_{\nu}(\beta v)<F_{\nu}(u)$. However, as this contradicts the assumptions on $u$, we obtain that $u$ is a minimizer of the action $S_{\nu}^{\omega}$ at frequency $\omega$.

Suppose now that $\nu<0$. By the Lagrange Multiplier Theorem, one has
 \[
  \langle F_{\nu}'(u),v\rangle+\omega\langle u,v\rangle=0,\qquad\forall v\in D_{\nu},
 \]
 so that, setting $v=u$, we get 
 \begin{equation}
 \label{eq-eneq}
 2F_{\nu}(u)-\f{p-2}{p}\|u\|_p^p=-\omega\|u\|_2^2.
 \end{equation}
 Therefore, by point $(iii)$ in Proposition \ref{prop:Classification} 
 \begin{equation*}
 -\omega=\f{2\F_{\nu}(\mu)-\f{p-2}{p}\|u\|_{p}^{p}}{\mu}<2\mu^{-1}\F_{\nu}(\mu)<\mu^{-1}\inf_{v\in \D_{\nu}^\mu}Q_{\nu}(v)=-\omega_{\nu},
 \end{equation*}
 entailing the thesis.
\end{proof}

\begin{proof}[Proof of Theorem \ref{ex-gs-mu}]
Let $u_{n}$ be a minimizing sequence of mass $\mu$ for $F_{\nu}$. Then by Lemma \ref{minimseq2} there exists $u\in \D_{\nu}\setminus H^{1}(\Rd)$ such that, up to subsequences, all the limits in \eqref{lim-1} hold and, for any $\la\geq 4\nu^{2}$ and given the decomposition $u_{n}=\phi_{n,\la}+q_{n}\G_{\la,\nu}$, also the limits in \eqref{lim-2} hold. Set $m:=\|u\|_{2}^{2}$. By weak lower semicontinuity of the norm, $m\leq \mu$. Since $q\neq 0$, it follows that $m\neq 0$. Suppose by contradiction that $0<m<\mu$. Since $u_{n}\deb u$ in $L^{2}(\Rd)$, then $\|u_{n}-u\|_{2}^{2}=\mu-m+o(1)$, as $n\to+\infty$. On the one hand, since $p>2$ and $\f{\mu}{\|u_{n}-u\|_{2}^{2}}>1$,
\begin{equation*}
\F_{\nu}(\mu)\leq F_{\nu}\left(\f{\sqrt{\mu}}{\|u_{n}-u\|_{2}}(u_{n}-u)\right)<\f{\mu}{\|u_{n}-u\|_{2}^{2}} F_{\nu}(u_{n}-u),
\end{equation*}
thus
\begin{equation}
\label{first-ineq-exgs}
\liminf_{n} F_{\nu}(u_{n}-u)\geq \f{\mu-m}{\mu} \F_{\nu}(\mu).
\end{equation}
On the other hand,
\begin{equation*}
\F_{\nu}(\mu)\leq F_\nu\left(\sqrt{\f{\mu}{m}}u\right)<\f{\mu}{m}F_{\nu}(u),
\end{equation*}
hence 
\begin{equation}
\label{sec-ineq-exgs}
F_{\nu}(u)>\f{m}{\mu}\F_{\nu}(\mu).
\end{equation}
Moreover, there results that
\begin{equation}
\label{Q-un-u}
Q_{\nu}(u_{n}-u)=Q_{\nu}(u_{n})-Q_{\nu}(u)+o(1),\quad n\to+\infty
\end{equation}
and
\begin{equation}
\label{np-un-u}
\|u_{n}-u\|_{p}^{p}=\|u_{n}\|_{p}^{p}-\|u\|_{p}^{p}+o(1), \quad n \to +\infty,
\end{equation}
hence
\begin{equation}
\label{BL-lemma}
F_{\nu}(u_{n}-u)=F_{\nu}(u_{n})-F_{\nu}(u)+o(1),\quad n\to+\infty.
\end{equation}

In particular, \eqref{Q-un-u} is a consequence of the fact that $u_{n}\deb u$, $\phi_{n,\la}\deb \phi_{\la}$, $\na \phi_{n,\la}\deb \na \phi_{\la}$ in $L^{2}(\R^{2})$, $q_{n}\to q$ in $\C$ and \eqref{phin-phi-Coul}.


Equation \eqref{np-un-u} follows instead by Brezis-Lieb Lemma \cite{Brezis-Lieb-83}, since $(u_{n})_{n}$ is bounded in $L^{p}(\Rd)$ and $u_{n}\to u$ a.e. in $\Rd$.

By combining \eqref{first-ineq-exgs}, \eqref{sec-ineq-exgs} and \eqref{BL-lemma}, we obtain
\begin{equation*}
\F_{\nu}(\mu)=\liminf_{n}F_{\nu}(u_{n})= \liminf_{n} F_{\nu}(u_{n}-u) +F_{\nu}(u)>\F_{\nu}(\mu),
\end{equation*}
that is a contradiction, thus $m=\mu$.

In order to conclude the proof of the existence part, it is left to prove that 
\begin{equation*}
\F_{\nu}(\mu)\leq \liminf_{n} F_{\nu}(u_{n}),
\end{equation*}
in particular that $u_{n}\to u$ in $L^{p}(\Rd)$ and $\Big\|\f{\phi_{n,\la}-\phi_{\la}}{\sqrt{|x|}}\Big\|_{2}\to 0$ as $n\to +\infty$. By applying \eqref{GN-stand}, we have that for every $r>2$ 
\begin{equation*}
\|u_{n}-u\|_{r}^{r}\leq 2^{r-1}\left(K_{r}\|\na\phi_{n,\la}-\na\phi_{\la}\|_{2}^{r-2}\|\phi_{n,\la}-\phi_{\la}\|_2^2+\|\G_{\la,\nu}\|_{r}^{r}|q_{n}-q|^{r}\right),
\end{equation*}
hence $u_{n}\to u$ in $L^{r}(\Rd)$ since $\|\na\phi_{n,\la}-\na\phi_{\la}\|_{2}$ is bounded, $\phi_{n,\la}\to \phi_{\la}$ strongly in $L^{2}(\Rd)$ and $q_{n}\to q$ in $\C$: this is valid in particular for $r=p$. Finally,
\begin{equation*}
\begin{split}
\left\|\f{\phi_{n,\la}-\phi_{\la}}{\sqrt{|x|}}\right\|_{2}^{2}&=\int_{B_{1}}\f{|\phi_{n,\la}-\phi_{\la}|^{2}}{|x|}\,\ud x+\int_{\Rd\setminus B_{1}}\f{|\phi_{n,\la}-\phi_{\la}|^{2}}{|x|}\,\ud x\\
&\leq\left(\int_{B_{1}}\f{1}{|x|^{\f{p}{p-1}}}\,\ud x\right)^{1-\f{1}{p}}\|\phi_{n,\la}-\phi_{\la}\|_{2p}^{2}+\|\phi_{n,\la}-\phi_{\la}\|_{2}^{2}\rightarrow 0,\quad n\to+\infty,
\end{split}
\end{equation*}
since $\phi_{n,\la}\to \phi_\la$ strongly both in $L^{2p}(\Rd)$ and in $L^{2}(\Rd)$ and
\begin{equation*}
\int_{B_{1}}\f{1}{|x|^{\f{p}{p-1}}}\,\ud x=2\pi\int_{0}^{1}\f{1}{\rho^{\f{1}{p-1}}}\,\ud x<+\infty.
\end{equation*}

We are ready now to prove the properties of the ground states at mass $\mu$. We observe that, by Lemma \ref{GSareAM}, every ground state $u$ at mass $\mu$ is a minimizer of the action at frequency $\omega>\omega_{\nu}$. Moreover, by Theorem \ref{ex-min-act}, $u$ presents a logarithmic singularity at the origin, is positive, radially symmetric and decreasing along the radial direction, up to gauge invariance, and this concludes the proof.
\end{proof}

\appendix

\section{Self-adjoint extensions of $2d$ Schr\"odinger operator with Coulomb potential}
\label{app:self-adj}

In this appendix we prove Proposition \ref{prop:Classification} using the Kre\u{\i}n-Vi\v{s}ik-Birman self-adjoint extension theory. Since it will be needed to distinguish among a self-adjoint operator and its quadratic form, we will make use of the following notation. If $A$ is a densely defined symmetric operator, $\mathcal{D}(A)$ is the operator domain, $A u$ denotes the action of $A$ on $u \in \mathcal{D}(A)$; $\mathcal{D}[A]$ denotes the domain of the quadratic form associated to $A$, whose action is denoted by $A[u,v]$ for $u,v \in \mathcal{D}[A]$.

\subsection{Unitary equivalent problem.}

To simplify the analysis, we exploit the spherical symmetry of the potential and the isomorphism given by the unitary operator
\begin{equation}\label{eq:UnitaryPolar}
	\begin{array}{ccccc}
		U & : & L^2(\mathbb{R}^2,\ud x \ud y) & \longrightarrow & \ell^2(\mathbb{Z},L^2(\mathbb{R}^+, \ud r)) \\
		& & f & \longmapsto & (\widehat{\varphi}_k(r))_k
	\end{array}
\end{equation}
where for $k \in \mathbb{Z}$
\begin{equation*}
	\widehat{\varphi}_k(r) \;:=\; \sqrt{\frac{r}{2 \pi}} \int_0^{2 \pi} e^{\ii 2 \pi k \vartheta} f(r, \vartheta) \, \ud \vartheta
\end{equation*}
and $r,\vartheta \in \mathbb{R}^+ \times [0,2\pi)$ are the polar coordinates on the plane
\begin{equation*}
	\begin{cases}
		x \; = \; r \, \cos \vartheta \\
		y \; = \; r \, \sin \vartheta
	\end{cases} \, .
\end{equation*}
The transformed Hamiltonian has the form
\begin{equation*}
	U H U^* \;=\; \bigoplus_{k \in \mathbb{Z}} \left[-\frac{\ud^2}{\ud r^2} + \left( (2 \pi k)^2 - \frac{1}{4} \right)\frac{1}{r^2}+\frac{\nu}{r} \right] \, .
\end{equation*}
A simple application of Weyl's limit circle/limit point criterion shows that the only component which is not essentially self-adjoint is the one corresponding to $k=0$. It is also trivial to see that the operator \eqref{eq:symm_op} is \emph{bounded from below}. We call $S$ the $s$-wave operator, that is
\begin{equation*}
	S\;:=\; -\frac{\ud^2}{\ud r^2} -\frac{1}{4} \frac{1}{r^2} + \frac{\nu}{r} \, .
\end{equation*}

\subsection{Kernel of the shifted operator}

 \label{subsec:Kernel}
The first ingredient needed for the application of the Kre\u{\i}n-Vi\v{s}ik-Birman extension scheme is the characterisation of the kernel of the operator $S^*$, i.e. to find all the solution to the differential problem
\begin{equation}\label{eq:KernelEq}
	\left( -\frac{\ud^2}{\ud r^2} -\frac{1}{4 r^2} + \frac{\nu}{r} +\lambda \right) u(r) \;=\; 0
\end{equation}
in $L^2(\mathbb{R}^+)$. Through the change of variables $r \mapsto \xi = 2 r \sqrt{\lambda}$, and looking for a solution in the form $u(r)=\phi(2 r \sqrt{\lambda})=\phi(\xi)$, \eqref{eq:KernelEq} reduces to the Whittaker's equation \cite[Eq. 13.1.31]{Abramowitz-Stegun-1964}
\begin{equation}
	\frac{\ud^2 \phi}{\ud \xi^2} + \left[-\frac{1}{4}-\frac{\nu}{2 \sqrt{\lambda}} \frac{1}{\xi}+ \frac{1}{4 \xi^2} \right] \phi \;=\; 0 \,.
\end{equation}
A pair of linearly independent solutions is given, for $\lambda > 0$, by
\begin{eqnarray}
	\Phi_{\nu,\lambda}(r)\!\!\!&:=& \!\!\!\frac{1}{\sqrt{2 \pi}} \Gamma\big({\textstyle \frac{1}{2}+\frac{\nu}{2\sqrt{\lambda}}}\big)\sqrt{r} e^{-\sqrt{\lambda} r} U_{\frac{1}{2}+\frac{\nu}{2 \sqrt{\lambda}},1}(2 \sqrt{\lambda} r)\, ,\label{eq:Phi} \\
	F_{\nu,\lambda}(r) \!\!\!&:=&  \!\!\!{\frac{\sqrt{2 \pi}}{ \Gamma^2\big({\textstyle \frac{1}{2}+\frac{\nu}{2\sqrt{\lambda}}}\big)}}\sqrt{r} e^{-\sqrt{\lambda} r} M_{\frac{1}{2}+\frac{\nu}{2 \sqrt{\lambda}},1}(2 \sqrt{\lambda} r)	\nonumber
\end{eqnarray}
where $M$ and $U$ are \emph{Kummer} and \emph{Tricomi} functions defined in \cite[Eq. 13.1.2, 13.1.6]{Abramowitz-Stegun-1964} and the large-distance asymptotics are, to leading order, given by \cite[13.5.1 - 13.5.2]{Abramowitz-Stegun-1964}
\begin{eqnarray}
	\Phi_{\nu,\lambda}(r) \!\!\!&=&\!\!\! \frac{1}{2 \sqrt{\pi\lambda}}{ \Gamma\left({\textstyle \frac{1}{2}+\frac{\nu}{2\sqrt{\lambda}}}\right)} \big(2 r \sqrt{\lambda} \big)^{-\frac{\nu}{2 \sqrt{\lambda}}} e^{-\sqrt{\lambda} r} \big(1+O(r^{-1}) \Big) \, , \nonumber\\
	F_{\nu,\lambda}(r) \!\!\!&=&\!\!\! \frac{\sqrt{\pi}}{\sqrt{\lambda} \, \Gamma{^2}(\frac{1}{2}+\frac{\nu}{2 \sqrt{\lambda}})} (2 \sqrt{\lambda} r)^{\frac{\nu}{2 \sqrt{\lambda}}} e^{\sqrt{\lambda} r} \big( 1 + O(r^{-1}) \big) \, ,\nonumber
\end{eqnarray}
showing that the only potentially square integrable solution is given by $\Phi_{\nu,\lambda}(r)$.

Concerning short-distance asymptotics, these are obtained from \cite[Eq. 13.1.2 - 13.1.6]{Abramowitz-Stegun-1964}
\begin{eqnarray}\label{eq:Asymptotics_Phi_zero}
	\Phi_{\nu,\lambda}(r) \!\!\!&=&\!\!\! - \sqrt{\frac{r}{2 \pi}} \bigg[\ln r+\psi\big({\textstyle\frac{1}{2}+\frac{\nu}{2 \sqrt{\lambda}} } \big)+2 \gamma + \ln( 2 \sqrt{\lambda}) + \nu r \ln r\bigg] + O(r^{3/2})	\, , \\
\label{eq:AsymptFzero}
	F_{\nu,\lambda}(r) \!\!\!&=&\!\!\! {\frac{\sqrt{2\pi}}{\Gamma^2(\frac{1}{2}+\frac{\nu}{2 \sqrt{\lambda}})}}\sqrt{r} \big(1+ \nu r\big) + O(r^{5/2}) \, ,
\end{eqnarray}
here $\psi(z)=\Gamma'(z)/\Gamma(z)$ is the digamma function and $\gamma \sim 0.577$ is the Euler-Mascheroni constant. Asymptotics \eqref{eq:Asymptotics_Phi_zero} shows that, in fact, $\Phi_{\nu,\lambda} \in L^2(\mathbb{R}^+)$ and we can summarise the result proved in this subsection as:

\begin{lemma}
	The operator \eqref{eq:symm_op} has deficiency index 1 and $\ker (S^*+\lambda\mathbbm{1})=\mathrm{span} \{\Phi_{\nu,\lambda}\}$ with $\Phi_{\nu,\lambda}$ defined in \eqref{eq:Phi}.
\end{lemma}

\subsection{Domain of the adjoint}
As a consequence of the method of variation of constants, to solve non-homogeneous ODEs we know the following. If we define the Wronskian of $\Phi_{\nu,\lambda}$ and $F_{\nu,\lambda}$ as
\begin{equation}
	W_{r}(\Phi_{\nu,\lambda},F_{\nu,\lambda}) \;:=\; \det \begin{pmatrix}
		\Phi_{\nu,\lambda}(r) & F_{\nu,\lambda}(r) \\
		\Phi'_{\nu,\lambda}(r) & F'_{\nu,\lambda}(r)
	\end{pmatrix} \;=\; \frac{1}{\Gamma(\frac{1}{2}+\frac{\nu}{2 \sqrt{\lambda}})}=:W_{\nu,\lambda}
\end{equation}
and the \emph{two-point Green's function} as
\begin{equation}\label{eq:Green}
	\G_{\lambda,\nu}(r,\rho) \;:=\; \frac{1}{W_{\nu,\lambda}}\begin{cases}
		\Phi_{\nu,\lambda}(r) F_{\nu,\lambda}(\rho) \qquad 0< \rho < r, \\
		\\
		F_{\nu,\lambda}(r) \Phi_{\nu,\lambda} (\rho) \qquad r < \rho < +\infty
	\end{cases} \, .
\end{equation}
 
We have
\begin{lemma}~ 
	\begin{enumerate}
		\item If $\lambda \in (-\infty,-\nu^2) $, the operator $R_{\G_{\lambda,\nu}}$ whose integral kernel is given by \eqref{eq:Green} is bounded, self-adjoint and everywhere defined.
		\item For any $g \in L^2(\mathbb{R}^+)$, the short-distance asymptotics of $(R_{\G_{\lambda,\nu}}g)(r)$ is
				\begin{equation}\label{eq:AsymptRG}
			(R_{\G_{\lambda,\nu}} g)(r) \;=\; \sqrt{2 \pi r} \langle \Phi_{\nu,\lambda}, g \rangle_{L^2(\mathbb{R}^+)} + o(r^{3/2} \ln r) \, .
		\end{equation}
		\item $R_{\G_{\lambda,\nu}}$ is the resolvent at $\lambda$ of a certain self-adjoint extension of $S$, called $S_D$: $(S_D+\lambda)^{-1}=R_{\G_{\lambda,\nu}}$.
	\end{enumerate}
	\begin{proof}
		(1) $\G_{\lambda,\nu}(r,\rho)$ is well-defined for $\lambda \in \mathbb{C}\setminus \mathbb{R}$ using as determination of square root the one with positive real part. This implies exponential decay of $\Phi_{\nu,\lambda}$ at large distances which is a condition we are going to use in the proof. If $\nu>0$ $\G_{\lambda,\nu}(r,\rho)$ is well defined for all $\lambda >0$; if $\nu < 0$ one has to take care of the poles of the Gamma function occurring for negative integers of its argument. The explicit expression \eqref{eq:Green} guarantees that a sufficient condition for $\G_{\lambda,\nu}(r,\rho)$ to be well defined at real $\lambda$, even in the case $\nu<0$, is $\lambda > \nu^2$.
		
		 The integral kernel of the operator $R_{\G_{\lambda,\nu}}$ splits into the sum of four integral operators with kernels given by
		\[
			\begin{split}
				\G_{\lambda,\nu}^{++}(r,\rho)\;&:=\; \G_{\lambda,\nu}(r,\rho) \mathbbm{1}_{(1,+\infty)}(r) \mathbbm{1}_{(1,+\infty)} (\rho) \\
				\G_{\lambda,\nu}^{+-}(r,\rho)\;&:=\; \G_{\lambda,\nu}(r,\rho) \mathbbm{1}_{(1,+\infty)}(r) \mathbbm{1}_{(0,1)} (\rho) \\
				\G_{\lambda,\nu}^{-+}(r,\rho)\;&:=\; \G_{\lambda,\nu}(r,\rho) \mathbbm{1}_{(0,1)}(r) \mathbbm{1}_{(1,+\infty)} (\rho) \\
				\G_{\lambda,\nu}^{++}(r,\rho)\;&:=\; \G_{\lambda,\nu}(r,\rho) \mathbbm{1}_{(0,1)}(r) \mathbbm{1}_{(0,1)} (\rho) \\				
			\end{split}
		\]
		where $\mathbbm{1}_J$ denotes the characteristic function supported on the interval $J \subset \mathbb{R}^+$. 
		
		For each $G^{LM}_{\nu,\lambda}(r,\rho)$ with $L,M \in \{+,-\}$ we can derive a point-wise estimate in $(r,\rho)$ from the large and short distance asymptotics of $\Phi_{\nu,\lambda}$ and $F_{\nu,\lambda}$. These are
		\[
			\begin{split}
				|\G_{\lambda,\nu}^{++}(r,\rho)| \; &\lesssim \; \Big({\textstyle \big(\frac{\rho}{r} \big)^{\frac{\nu}{2 \sqrt{\lambda}}}+\big(\frac{r}{\rho} \big)^{\frac{\nu}{2 \sqrt{\lambda}}}} \Big) e^{-\sqrt{\lambda} |r- \rho|}\mathbbm{1}_{(1,+\infty)}(r) \mathbbm{1}_{(1,+\infty)} (\rho) \\
			|\G_{\lambda,\nu}^{+-}(r,\rho)| \; & \lesssim \; \sqrt{\rho} r^{-\frac{\nu}{2 \sqrt{\lambda}}} e^{-\sqrt{\lambda} r}\mathbbm{1}_{(1,+\infty)}(r) \mathbbm{1}_{(0,1)} (\rho) \\
			|\G_{\lambda,\nu}^{-+}(r,\rho)| \; & \lesssim \; \sqrt{r} \rho^{-\frac{\nu}{2 \sqrt{\lambda}}} e^{-\sqrt{\lambda} \rho}\mathbbm{1}_{(0,1)}(r) \mathbbm{1}_{(1,+\infty)} (\rho) \\
			|\G_{\lambda,\nu}^{--}(r,\rho)| \; & \lesssim \; \big(\sqrt{\rho} \ln r + \sqrt{r} \ln \rho \big)\mathbbm{1}_{(0,1)}(r) \mathbbm{1}_{(0,1)} (\rho
			\end{split}
		\]
		The last three estimates show that the kernels $G^{+-}_{\nu,\lambda}(r,\rho)$, $G^{-+}_{\nu,\lambda}(r,\rho)$ and $G^{--}_{\nu,\lambda}(r,\rho)$ are in $L^2(\mathbb{R}^+\times \mathbb{R}^+,\ud r \ud \rho)$ and therefore the corresponding integral operator are Hilbert-Schmidt and, in particular, bounded. A standard Schur test and the first estimate allows to conclude that the integral operator with kernel $G^{++}_{\nu,\lambda}(r,\rho)$ is bounded and thus the overall boundedness of $R_{\G_{\lambda,\nu}}$. 
		
		If $\lambda \in \mathbb{R}$ the operator is bounded and symmetric (as its integral kernel is symmetric) and therefore also self-adjoint.
		
		(2) To estimate $(R_{\G_{\lambda,\nu}}g)(r)$ we use the explicit expression of $\G_{\lambda,\nu}$ together with the fact that $\Phi_{\nu,\lambda} \in L^2(\mathbb{R}^+)$. This yields
		\[
			\begin{split}
				(R_{\G_{\lambda,\nu}}g)(r) \;&=\; \frac{1}{W_{\nu,\lambda}} \Big[\Phi_{\nu,\lambda}(r) \int_0^r F_{\nu,\lambda}(\rho) g(\rho) \, \ud \rho + F_{\nu,\lambda}(r) \int_r^{+\infty} \Phi_{\nu,\lambda}(\rho) g(\rho) \, \ud \rho \Big] \\
				&=\frac{1}{W_{\nu,\lambda}} F_{\nu,\lambda}(r) \int_0^{+\infty} \Phi_{\nu,\lambda}(\rho) g(\rho) \, \ud \rho + \\
				& \quad +\frac{1}{W_{\nu,\lambda}} \Big[\Phi_{\nu,\lambda}(r) \int_0^r F_{\nu,\lambda}(\rho) \, g(\rho)\, \ud \rho -F_{\nu,\lambda}(r) \int_0^r\Phi_{\nu,\lambda}(\rho) \, g(\rho) \, \ud \rho \Big] 
			\end{split}
		\]
		
		We can now use asymptotic expansion \eqref{eq:AsymptFzero} to get as $r \downarrow 0$
		\[
			\begin{split}
				\frac{1}{W_{\nu,\lambda}} F_{\nu,\lambda}(r) \langle \Phi_{\nu,\lambda}, g \rangle_{L^2(\mathbb{R}^+)} \;&=\; \frac{\sqrt{2 \pi r}}{W_{\nu,\lambda} \Gamma(\frac{1}{2}+\frac{\nu}{2 \sqrt{\lambda}})} \langle \Phi_{\nu,\lambda}, g \rangle_{L^2(\mathbb{R}^+)} + O(r^{3/2}) \\
				&=\sqrt{2 \pi r} \langle \Phi_{\nu,\lambda},g \rangle + O(r^{3/2}) \, .
			\end{split}
		\]
		Using Schwarz inequality and asymtptotics \eqref{eq:AsymptFzero}-\eqref{eq:Asymptotics_Phi_zero} we can estimate the other term:
		\[
			\begin{split}
				\big|\Phi_{\nu,\lambda}(r)\big| \int_0^r \big|F_{\nu,\lambda}(\rho)\big| \big|g(\rho)\big| \, \ud \rho\; &\lesssim \; \sqrt{r} \ln r \Vert F_{\nu,\lambda} \Vert_{L^2(0,r)} \Vert g \Vert_{L^2(0,r)} \\
				&\sqrt{r} \ln r \Vert \sqrt{\rho} \Vert_{L^2(0,r)} o(1) \;=\; o(r^{3/2} \ln r) \, ;
			\end{split}
		\]
		\[
			\begin{split}
				\big|F_{\nu,\lambda}(r)\big| \int_0^r \big| \Phi_{\nu,\lambda}(\rho)\big| \big|g(\rho) \big| \, \ud \rho \; &\lesssim \; \sqrt{r} \Vert \sqrt{\rho} \ln \rho \Vert_{L^2(0,r)} \Vert g \Vert_{L^2(0,r)} = o(r^{3/2} \ln r) \, .
			\end{split}
		\]
		Since $O(r^{3/2})=o(r^{3/2} \ln r)$, this completes the proof of (2).
		
		(3) To see that $R_{\G_{\lambda,\nu}}$ inverts a self-adjoint extension of $S+\lambda$ it is sufficient to notice that for every $g \in C^\infty_0(\mathbb{R}^+)$ we have $R_{\G_{\lambda,\nu}} (S+\lambda)g=(S+\lambda)R_{\G_{\lambda,\nu}}g=g$ and since $R_{\G_{\lambda,\nu}}$ on $C^\infty_0(\mathbb{R}^+)$ extends to a unique self-adjoint operator on $L^2(\mathbb{R}^+)$ then it is the inverse of a densely defined self-adjoint operator that extends $(S+\lambda)$.
	\end{proof}
\end{lemma}

In particular, defining $\Psi_{\nu, \lambda}:=R_{\G_{\lambda,\nu}} \Phi_{\nu,\lambda}$, this Lemma implies that
\begin{equation}\label{eq:AsymptoticPsi}
	\Psi_{\nu,\lambda}(r)\;=\;(R_{\G_{\lambda,\nu}} \Phi_{\nu,\lambda})(r) \;=\; \sqrt{2 \pi r} \Vert \Phi_{\nu,\lambda}\Vert_{L^2(\mathbb{R}^+)}^2 +  o(r^{3/2} \ln r)
\end{equation}

With these results at hand, by Kre\u{\i}n-Vi\v{s}ik-Birman characterisation of $\mathcal{D}(S^*)$ \cite[Theorem 1]{GMO-KVB2017} we know that a generic function $g \in \mathcal{D}(S^*)$ has the following structure 
\begin{equation}\label{eq:FunctionAdjoint}
	g\;=\; c_0 \Phi_{\nu,\lambda}+c_1 \Psi_{\nu,\lambda} + h, \qquad c_0,c_1 \in \mathbb{C} \, , \, \, h \in \mathcal{D}(\overline{S}) \, .
\end{equation}

To compute its asymptotic expansion as $r \downarrow 0$ we need some information on $h$. The operator $S$ corresponds to the operator $L_{\beta,\alpha}$ in \cite[Sec 3.2]{Derezinski-Richard-2017} for $\beta=\nu$ and $\alpha=0$. Thus, as stated in \cite[Proposition 3.1.]{Derezinski-Richard-2018}, if $h \in \mathcal{D}(\overline{S})$ then $h,h'$ are continuous on $\mathbb{R}^+$ and displays the following asymptotics as $r \downarrow 0$:
\begin{equation*}
	h(r)\;=\;o(r^{3/2}\ln r) \, , \qquad h'(r) \;=\; o(r^{1/2} \ln r) \, .
\end{equation*}

We can therefore compute the asymptotics of $g \in \mathcal{D}(S^*)$ as:
\begin{equation}\label{eq:Asymtptgr}
	g(r) \;=\; \sqrt{\frac{r}{2 \pi}} \left\{-\ln r c_0 + \left[ 2 \pi \Vert \Phi_{\nu,\lambda} \Vert_{L^2(\mathbb{R}^+)}^2 c_1 - \left(\psi\big({\textstyle\frac{1}{2}+\frac{\nu}{2 \sqrt{\lambda}}}\big)+2 \gamma+\ln (2 \sqrt{\lambda}) \right) c_0  \right] + O(r \ln r)\right\}
\end{equation}

\subsection{Self-adjoint extensions of the radial operator.} 
Kre\u{\i}n-Vi\v{s}ik-Birman extension theory \cite[Theorem 5]{GMO-KVB2017}, applied to the present case of deficiency index one, states that self-adjoint extensions of $S$ correspond to restrictions of $S^*$ to subspaces of $\mathcal{D}(S^*)$ that, in terms of formula \eqref{eq:FunctionAdjoint}, are identified by the condition
\begin{equation}\label{eq:BetaCondition}
	c_1 \;= \; \beta c_0 \, \qquad \text{for some $\beta \in \mathbb{R} \cup \{\infty\}$} \, .
\end{equation}
Conventionally, the extension parametrised by $\beta=\infty$ has domain defined by the condition $c_0=0$. We now characterise the domain of self-adjointness in a more convenient way.

We start by defining
\begin{eqnarray}
	g_0&:=& -\lim_{r \downarrow 0} \frac{1}{\sqrt{r} \ln r} g(r) \, , \nonumber\\
	g_1&:=& \lim_{r \downarrow 0} \frac{1}{\sqrt{r}}\Big\{g(r)+g_0 \sqrt{r} \ln r \Big\} \nonumber\, 
\end{eqnarray}
that means rewriting the asymptotics of $g(r) \in \mathcal{D}(S^*)$ as
\begin{equation}
	g(r)=-g_0 \sqrt{r} \ln r + g_1 \sqrt{r} + O(r^{3/2} \ln r)
\end{equation}

Let us now denote with $S_\beta$ the extension of $S$ characterised by the condition \eqref{eq:BetaCondition} for a given $\beta$. Using \eqref{eq:Asymtptgr} and \eqref{eq:BetaCondition} for  $g \in \mathcal{D}(S_\beta)$ one gets
\begin{eqnarray}
	g_0\!\!\!&=&\!\!\!\frac{c_0}{\sqrt{2\pi}} \, , \nonumber\\
	g_1\!\!\!&=&\!\!\!\bigg\{ - \frac{1}{\sqrt{2 \pi} } \Big[ \psi\big({\textstyle \frac{1}{2}+\frac{\nu}{2 \sqrt{\lambda}}}\big)+2 \gamma + \ln(2 \sqrt{\lambda}) \Big]+ \sqrt{2 \pi} \Vert \Phi_{\nu,\lambda}\Vert^2_{L^2(\mathbb{R}^+)}\beta \bigg\} c_0 \, . \nonumber
\end{eqnarray}
Upon defining $\alpha$ as
\begin{equation}\label{eq:AlphaBeta}
	\alpha_\beta \;:=\; \Vert \Phi_{\nu,\lambda} \Vert^2_{L^2(\mathbb{R}^+)} \beta - \frac{1}{2 \pi} \left( \psi\big({\textstyle \frac{1}{2}+\frac{\nu}{2 \sqrt{\lambda}}} \big)+2\gamma+\ln(2 \sqrt{\lambda}) \right)
\end{equation}
we can rewrite the domain of the self-adjoint $S_\beta$ as
\begin{equation}\label{eq:DomainSalpha}
	\mathcal{D}(S_\beta) \;=\; \big\{g \in \mathcal{D}(S^*) \; | \; g_1 = 2 \pi \alpha_\beta g_0 \big\} \, .
\end{equation}	
Since the relation $\alpha_\beta \leftrightarrow \beta$ is one-to-one, we can denote as $S_\alpha$ the self-adjoint extension whose domain is given by \eqref{eq:DomainSalpha}. Thus, explicitly
\begin{equation*}
	\mathcal{D}(S_\alpha)\;=\; \big\{ g \in \mathcal{D}(S^*) \, | \, g_1 \;=\; 2 \pi \alpha g_0 \big\} \, .
\end{equation*}

Using this last characterisation of the domain it is immediate to identify $S_D=S_F$ \cite[Proposition 1]{2010-Duclos-2D-Delta-Coulomb}.

We now move to the general classification of the quadratic forms associated to self-adjoint extensions of $S+\lambda$. Owing to \cite[Theorem 7]{GMO-KVB2017}, for all extensions but the Friedrichs we have
\begin{equation*}
	\mathcal{D}[S_\alpha+\lambda]\;=\; \mathcal{D}[S_F] \dot{+} \mathrm{span} \{\Phi_{\nu,\lambda} \} \, .
\end{equation*}
\begin{equation*}
	(S_{\alpha}+\lambda)\big[f+q \Phi_{\nu,\lambda},f'+q' \Phi_{\nu,\lambda} \big] \;=\; (S_\alpha+\lambda)[f,f']+\beta_\alpha \overline{q} q' \Vert \Phi_{\nu,\lambda} \Vert^2_{L^2(\mathbb{R}^+)} \, .
\end{equation*}
Calling $g=f+q \Phi_{\nu,\lambda}$ and $g'=f'+q' \Phi_{\nu,\lambda}$ and using the definition of $\theta_{\nu,\lambda}$  (see \eqref{eq:thetaDef}) one has
\begin{equation*}
	\begin{split}
		S_\alpha[g,g']\;&=\; (S_\alpha+\lambda)[g,g']-\lambda \langle g,g' \rangle \\
		&=\;(S_F+\lambda)[f,f']+\beta_\alpha \overline{q} q' \Vert \Phi_{\nu,\lambda} \Vert^2_{L^2(\mathbb{R}^+)} - \lambda\langle g, g' \rangle \\
		&=\;S_F[f,f']+\lambda \Big(\langle f,f'\rangle-\langle g,g' \rangle \Big)+ \big(\alpha+\theta_{\nu,\lambda} \big) \overline{q}q'
	\end{split}
\end{equation*}
whence the result
\begin{equation}\label{eq:RadialQuadraticForm}
	S_\alpha[g]\;=\; S_F[f]+\lambda \big(\Vert f \Vert^2_{L^2(\mathbb{R}^+)}-\Vert g \Vert_{L^2(\mathbb{R}^+)}^2 \big)+ (\alpha+\theta_{\nu,\lambda} ) |q|^2 \, .
\end{equation}

\subsection{Reconstruction of the 2D problem}

Now we put together the analysis of the previous subsection to prove the main results.

\begin{proof}[Proof of Proposition \ref{prop:Classification}]
Since for $k \neq 0$ the operator in the $k$-wave sector is essentially self-adjoint we can limit our discussion to the $s$-wave sector.

Suppose $g \in \mathcal{D}(H_{\nu,\alpha})$, then the $s$-wave component of $g$ has the form
\[
	\hat g_0(r) \;=\; \frac{\sqrt{r}}{\sqrt{2 \pi}} \int_0^{2 \pi} g(r,\vartheta) \ud \vartheta = \sqrt{r} \hat f_0(r)+ q \Phi_{\lambda,\nu}(r)
\]
Let us compute its asymptotics as $r \downarrow 0$ noting first that $f(0)= \frac{1}{\sqrt{2\pi}}\hat{f}_0(0)$. Then, by direct comparison with \eqref{eq:FunctionAdjoint} and using the asymptotic expansion \eqref{eq:AsymptoticPsi} one has
\[
	f(0)\;=\; \frac{1}{\sqrt{2 \pi r}} \hat{f}_0(0) \;=\; \Vert \Phi_{\nu,\lambda} \Vert^2_{L^2(\mathbb{R}^+)} c_1
\]
Imposing now the self-adjointness condition \eqref{eq:BetaCondition} with $c_0=q$ one has
\[
	f(0)\;=\; (\alpha+\theta_{\nu,\lambda}) q
\]
which is precisely the condition \eqref{eq:Domain2DExtension-OP}.
%

(ii) Since deficiency index of the minimal operator is one, all self-adjoint extensions are bounded from below (this follows from the general result \cite[Corollary 1]{GMO-KVB2017}). Therefore, \cite[Theorem 7]{GMO-KVB2017} applies as a classification result for the quadratic form associated to each self-adjoint extension. Indeed, for $\alpha \neq \infty$,
\[
	\mathcal{D}[H_{\nu,\alpha}] \;=\; \mathcal{D}[H_F] \dot{+} \mathrm{span}_{\mathbb{C}} \{\G_{\lambda,\nu}\} \, .
\]
It is sufficient to show that $\mathcal{D}[H_F]=H^1(\mathbb{R}^2)$. This follows from Kato inequality \cite[Theorem 1.3]{2002-Bouzouina} and KLMN Theorem \cite[Theorem X.17]{rs2}.  Indeed, $1/|x|$ potential is infinitesimally form-bounded with respect to the free Laplacian (see Lemma \ref{lem:Quasi-Hardy}). Concerning the evaluation of the quadratic form \eqref{eq:FormValue}, it follows from \eqref{eq:RadialQuadraticForm}.

By definition, the energy form is half of the value of the quadratic form: $Q(g)=\frac{1}{2}H_\alpha[g,g]$, which concludes the proof.
\end{proof}

Last, before concluding the appendix, let us give a couple of interesting remarks.

First, the classification of Proposition \ref{prop:Classification} is particularly useful to characterize eigenvalues of the self-adjoint extension determined by the parameter $\alpha$. Since for $\ell \neq 0$, all operators have the same spectrum, it is of interest to compute eigenvalues for the $\ell=0$ sector only (the so-called ``$s$-wave sector). Analogously to \cite{GM-hydrogenoid-2018,GM-2017-DC-EV} we have the following corollary.
\begin{corollary}\label{cor:Eigenvalues}
	$E$ is an eigenvalue of $H_{\nu,\alpha}$ (in the $s$-wave sector) if and only if $-E>0$ is a solution to 
	\begin{equation}\label{eq:eigenvaluesCondition}
		\alpha+\theta_{-E,\nu}=0 \, .
	\end{equation}
	For fixed $\alpha \in \mathbb{R}$ and $\nu<0$, \eqref{eq:eigenvaluesCondition} has infinitely many solutions $E_n<0$ accumulating at $E=0$.
	
	In particular, the eigenvalues $E_n^{(F)}$ of the Friedrichs extension $H_{F}$ exist if and only if $\nu<0$ and are given by
	\begin{equation}\label{eq:FriedrichsEV}
		E_n^{(F)} \;=\; -\frac{\nu^2}{(1+2n)^2} \, , \qquad n\in \mathbb{N}.
	\end{equation}
\end{corollary}
\begin{proof}
	Computation of eigenvalues for the self-adjoint extension parametrised by $\alpha \in \mathbb{R} \cup \{\infty\}$ amounts at solving the differential equation
	\[
	S^* g_E \;=\; E g_E
	\]
	with the boundary condition $g_1 = 2 \pi \alpha g_0$. For $E>0$, the space of solutions of the differential problem $S^*g=Eg$ do not admit square integrable solutions, thus there are no positive eigenvalues. For $E<0$, analysis of subsection \ref{subsec:Kernel}, identifies as only square integrable solution the function $g_E=\Phi_{\nu,-E}$ with $\Phi_{\nu,\lambda}$ defined in \eqref{eq:Phi}.
	
	It remains to check for which values of $E$, $g_E$ satisfies the boundary condition $g_1=2 \pi \alpha g_0$. Notice that
	\[
	g_E(r) \;=\;- \sqrt{r} \ln r -2 \pi \theta_{\nu,-E} \sqrt{r} + o(r^{3/2} \ln r)
	\]
	whence, boundary condition reads precisely \eqref{eq:eigenvaluesCondition}.
	
	For $\alpha=0$ the condition is equivalent to $\theta_{\nu,-E}=\infty$, that is equivalent to look for the positive values of $E$ for which $\theta_{\nu,-E}$ has vertical asymptotes. Since Gamma function has no zeros, only vertical asymptotes of $\theta_{\nu,-E}$ comes from the vertical asymptotes of the digamma function, that are
	\[
	\frac{1}{2}+\frac{\nu}{2 \sqrt{-E}} \;=\; -n \, , \qquad n \in \mathbb{N}
	\]
	which is equivalent to
	\[
	\frac{\nu}{\sqrt{-E}}=-2n-1 \, .
	\]
	This last equation has solution if and only if $\nu < 0$, and in such case it yields \eqref{eq:FriedrichsEV}
\end{proof}

Second, it may be of interest to relate the results with Coulomb-like potential to the ones without Coulomb potential. This is recovered as a `zero-charge limit', $\nu \to 0$. For the linear operator, one has the following remark.

\begin{remark}
	If $\nu \to 0$ the following limits hold true:
	\begin{equation*}
		\Gamma\big({\textstyle \frac{1}{2}+\frac{\nu}{2 \sqrt{\lambda}}} \big) \longrightarrow \Gamma(1/2)=\sqrt{\pi}
	\end{equation*}
	\begin{equation*}
		\lim_{\nu \to 0} U_{\frac{1}{2}+\frac{\nu}{2 \sqrt{\lambda}},1}(2 \sqrt{\lambda} r) = \frac{K_0(\sqrt{\lambda} r)}{\sqrt{\pi}} e^{\sqrt{\lambda} r}
	\end{equation*}
	and $\psi(1/2)=-2 \ln 2 - \gamma$. 
	
	One sees that Proposition \ref{prop:Classification} reduces to the analogous one for the standard $-\Delta+\text{``}\delta\text{''}$ on $L^2(\mathbb{R}^2)$ in the limit $\nu \to 0$ (see, e.g. Section 1.2 in \cite{Adami-Boni-Carlone-Tentarelli-2022}).
\end{remark}

\section{Ground states, minimizers of the action and bound states}
\label{app-gbstates}

In this section, we show that both ground states of \eqref{en-fan} and minimizers of the action \eqref{S-nu-om} are bound states, i.e. they satisfy \eqref{bcond} and \eqref{EL-eq}.

First, we note that, if $u$ is either a ground state of \eqref{en-fan} of mass $\mu$ or a minimizer of the action \eqref{S-nu-om} at frequency $\omega$, then it satisfies, for any fixed $\la>\nu^{2}\mathds{1}_{(-\infty,0)}$,
\begin{multline}
\label{ELdelta}
\langle\na\chi_{\la} ,\na\phi_{\la}\rangle+\nu\langle\chi_{\la}, |x|^{-1}\phi_{\la}\rangle+\la\langle\chi_{\la} ,\phi_{\la}\rangle+(\omega-\la) \langle \chi,u\rangle+\bar{\xi}q\left(\alpha+\theta_{\la,\nu}\right)-\langle\chi,|u|^{p-2}u\rangle=0\\
\forall\chi=\chi_{\la}+\xi\G_{\la,\nu}\in D_{\nu}.
\end{multline}
If $u$ is ground state of \eqref{en-fan} of mass $\mu$, then $\omega=\mu^{-1}(\|u\|_p^p-Q_{\nu}(u))$. Now, letting $\xi=0$ in \eqref{ELdelta}, so that $\chi=\chi_\la\in H^1(\Rd)$, there results
\begin{equation}
\langle\na\chi ,\na\phi_{\la}\rangle+\langle\chi ,\nu|x|^{-1}\phi_{\la}+\omega\phi_{\la}+(\omega-\la)q\G_{\la,\nu}-|u|^{p-2}u\rangle=0\qquad\forall\chi\in H^{1}(\Rd).
\end{equation}
Hence, as $\nu|x|^{-1}\phi_{\la}+\omega\phi_{\la}+(\omega-\la)q\G_{\la}-|u|^{p-2}u\in L^2(\Rd)$, $\phi_{\la}\in H^{2}(\Rd)$ and, by density, 
\begin{equation}
\label{EL2}
-\lap\phi_{\la}+\nu|x|^{-1}\phi_{\la}+\omega\phi_{\la}+(\omega-\la)q\G_{\la}-|u|^{p-2}u=0\qquad\text{in}\quad L^2(\Rd), 
\end{equation}
which is equivalent to \eqref{EL-eq}. On the other hand, letting $\chi_{\la}=0$ and $\xi=1$ in \eqref{ELdelta}, so that $\chi=\G_{\la,\nu}$, there results
\begin{equation}
 \langle \G_{\la,\nu}, (\omega-\la)u-|u|^{p-2}u\rangle+q\left(\alpha+\theta_{\la,\nu}\right)=0.
\end{equation} 
Finally, using \eqref{EL2}, we obtain
\begin{equation}
 \langle \G_{\la,\nu}, (-\lap+\nu|x|^{-1}+\la)\phi_{\la}\rangle=q\left(\alpha+\theta_\la\right),
\end{equation}
which is equivalent to $\phi_{\la}(0)=q\left(\alpha+\theta_{\la,\nu}\right)$, thus also \eqref{bcond} is satisfied.


\begin{thebibliography}{10}

\bibitem{Abramowitz-Stegun-1964}
{\sc M.~Abramowitz and I.~A. Stegun}, {\em {Handbook of Mathematical Functions
  with Formulas, Graphs, and Mathematical Tables}}, vol.~55 of {National Bureau
  of Standards Applied Mathematics Series}, For sale by the Superintendent of
  Documents, U.S. Government Printing Office, Washington, D.C., 1964.

\bibitem{ABCT-3d}
{\sc R.~Adami, F.~Boni, R.~Carlone, and L.~Tentarelli}, {\em Existence,
  Structure, and Robustness of Ground States of a {NLSE} in 3d with a Point
  Defect}, Journal of Mathematical Physics, 63 (2022), p.~071501.

\bibitem{Adami-Boni-Carlone-Tentarelli-2022}
{\sc R.~Adami, F.~Boni, R.~Carlone, and L.~Tentarelli}, {\em Ground States for
  the Planar {NLSE} with a Point Defect as Minimizers of the Constrained
  Energy}, Calculus of Variations and Partial Differential Equations, 61
  (2022).

\bibitem{Adami-Rev}
{\sc R.~Adami, F.~Boni, and A.~Ruighi}, {\em Non-Kirchhoff Vertices and
  Nonlinear Schr\"{o}dinger Ground States on Graphs}, Mathematics, 8 (2020),
  p.~617.

\bibitem{Adami-Noja-Visc-13}
{\sc R.~Adami, D.~Noja, and N.~Visciglia}, {\em Constrained Energy Minimization
  and Ground States for {NLS} with Point Defects}, Discrete \& Continuous
  Dynamical Systems - B, 18 (2013), pp.~1155--1188.

\bibitem{Alzer-97}
{\sc H.~Alzer}, {\em On some Inequalities for the Gamma and Psi Functions},
  Mathematics of Computation, 66 (1997), pp.~373--389.

\bibitem{Benguria-86}
{\sc R.~Benguria and L.~Jeanneret}, {\em Existence and Uniqueness of Positive
  Solutions of Semilinear Elliptic Equations with Coulomb Potentials on $\mathbb{R}^3$},
  Communications in Mathematical Physics, 104 (1986), pp.~291--306.

\bibitem{2002-Bouzouina}
{\sc A.~Bouzouina}, {\em Stability of the Two-Dimensional
  Brown-Ravenhall Operator}, 132 (2002), pp.~1133--1144.

\bibitem{Bransden2003-yf}
{\sc B.~H. Bransden and C.~J. Joachain}, {\em Physics of Atoms and Molecules},
  Prentice-Hall, London, England, 2~ed., Apr. 2003.

\bibitem{Brezis-Lieb-83}
{\sc H.~Br{\'{e}}zis and E.~Lieb}, {\em A Relation Between Pointwise
  Convergence of Functions and Convergence of Functionals}, Proceedings of the
  American Mathematical Society, 88 (1983), pp.~486--490.

\bibitem{Cac-Finco-Noja-21}
{\sc C.~Cacciapuoti, D.~Finco, and D.~Noja}, {\em Well Posedness of the
  Nonlinear Schr\"{o}dinger Equation with Isolated Singularities}, Journal of
  Differential Equations, 305 (2021), pp.~288--318.

\bibitem{Cac-Finco-Noja-22}
{\sc C.~Cacciapuoti, D.~Finco, and D.~Noja}, {\em Failure of Scattering for the
  NLSE with a Point Interaction in Dimension Two and Three}, 2022.

\bibitem{Cao-Malomed-95}
{\sc X.~D. Cao and B.~A. Malomed}, {\em Soliton-defect Collisions in the
  Nonlinear Schr\"{o}dinger Equation}, Physics Letters A, 206 (1995),
  pp.~177--182.

\bibitem{Caze-Lions-82}
{\sc T.~Cazenave and P.~L. Lions}, {\em Orbital Stability of Standing Waves for
  Some Nonlinear Schr\"odinger Equations}, Communications in Mathematical
  Physics, 85 (1982), pp.~549--561.

\bibitem{2009-deOliveira-Verri}
{\sc C.~R. de~Oliveira and A.~A. Verri}, {\em Self-adjoint Extensions of
  Coulomb Systems in 1, 2 and 3 Dimensions}, 324 (2009), pp.~251--266.

\bibitem{Derezinski-Richard-2017}
{\sc J.~Derezi{\'n}ski and S.~Richard}, {\em {On {S}chr{\"o}dinger Operators
  with Inverse Square Potentials on the Half-line}}, Ann. Henri Poincar{\'e},
  18 (2017), pp.~869--928.

\bibitem{Derezinski-Richard-2018}
{\sc J.~Derezi{\'{n}}ski and S.~Richard}, {\em On Radial Schr\"{o}dinger
  Operators with a Coulomb Potential}, 19 (2018), pp.~2869--2917.

\bibitem{Dinh-2021}
{\sc V.~D. Dinh}, {\em On Nonlinear Schr\"{o}dinger Equations with Attractive
  Inverse-power Potentials}, Topological Methods in Nonlinear Analysis,
  (2021), p.~1.

\bibitem{Dov-Serra-Tilli-22}
{\sc S.~Dovetta, E.~Serra, and P.~Tilli}, {\em Action Versus Energy Ground
  States in Nonlinear Schr\"{o}dinger Equations}, Mathematische Annalen, 385
  (2022), pp.~1545--1576.

\bibitem{2010-Duclos-2D-Delta-Coulomb}
{\sc P.~Duclos, P.~{\v{S}}{\v{t}}ov{\'{\i}}{\v{c}}ek, and M.~Tu{\v{s}}ek}, {\em
  On the Two-dimensional Coulomb-like Potential with a Central Point
  Interaction}, 43 (2010), p.~474020.

\bibitem{Edmunds1999}
{\sc D.~E. Edmunds and H.~Triebel}, {\em Sharp Sobolev Embeddings and Related
  Hardy Inequalities: The Critical Case}, Mathematische Nachrichten, 207
  (1999), pp.~79--92.

\bibitem{Finco-noja-22}
{\sc D.~Finco and D.~Noja}, {\em Blow-up and Instability of Standing Waves for
  the NLS with a Point Interaction in Dimension Two}, 2022.

\bibitem{FGI-22}
{\sc N.~Fukaya, V.~Georgiev, and M.~Ikeda}, {\em On Stability and Instability
  of Standing Waves for 2d-nonlinear Schr\"{o}dinger Equations with Point
  Interaction}, Journal of Differential Equations, 321 (2022), pp.~258--295.

\bibitem{Fuk-Ohta-Ozawa-08}
{\sc R.~Fukuizumi, M.~Ohta, and T.~Ozawa}, {\em Nonlinear Schr\"{o}dinger
  Equation with a Point Defect}, Annales de l{\textquotesingle}Institut Henri
  Poincar{\'{e}} C, Analyse non lin{\'{e}}aire, 25 (2008), pp.~837--845.

\bibitem{Flp-00}
{\sc T.~F\"{u}l\"{o}p and I.~Tsutsui}, {\em A Free Particle on a Circle with
  Point Interaction}, Physics Letters A, 264 (2000), pp.~366--374.

\bibitem{GM-2017-DC-EV}
{\sc M.~Gallone and A.~Michelangeli}, {\em {Discrete Spectra for Critical
  {D}irac-{C}oulomb {H}amiltonians}}, J. Math. Phys., 59 (2018), pp.~062108,
  19.

\bibitem{GM-hydrogenoid-2018}
{\sc M.~Gallone and A.~Michelangeli}, {\em {Hydrogenoid Spectra
  with Central Perturbations}}, Rep. Math. Phys., 84 (2019), pp.~215--243.

\bibitem{Gallone-Michelangeli-Book}
{\sc M.~Gallone and A.~Michelangeli}, {\em Self-Adjoint Extension Schemes and
  Modern Applications to Quantum Hamiltonians}, Springer International
  Publishing, 2023.

\bibitem{GMO-KVB2017}
{\sc M.~Gallone, A.~Michelangeli, and A.~Ottolini}, {\em {Kre{\u
  \i}n-Vi\v{s}ik-Birman Self-adjoint Extension Theory Revisited}}, in
  {Mathematical Challenges of Zero Range Physics}, A.~Michelangeli, ed.,
  {INdAM-Springer series, Vol.~42}, Springer International Publishing, 2020,
  pp.~239--304.

\bibitem{GMS-22}
{\sc V.~Georgiev, A.~Michelangeli, and R.~Scandone}, {\em Standing Waves and
  Global Well-posedness for the 2d Hartree Equation with a Point Interaction},
  2022.

\bibitem{GSS-87}
{\sc M.~Grillakis, J.~Shatah, and W.~Strauss}, {\em Stability Theory of
  Solitary Waves in the Presence of Symmetry, I}, Journal of Functional
  Analysis, 74 (1987), pp.~160--197.

\bibitem{Holm-Marz-Zwo-07}
{\sc J.~Holmer, J.~Marzuola, and M.~Zworski}, {\em Fast Soliton Scattering by
  Delta Impurities}, Communications in Mathematical Physics, 274 (2007),
  pp.~187--216.

\bibitem{Jeanjean-22}
{\sc L.~Jeanjean and S.-S. Lu}, {\em On Global Minimizers for a Mass
  Constrained Problem}, Calculus of Variations and Partial Differential
  Equations, 61 (2022).

\bibitem{Jeblick2019}
{\sc M.~Jeblick, N.~Leopold, and P.~Pickl}, {\em Derivation of the Time
  Dependent Gross{\textendash}Pitaevskii Equation in Two Dimensions},
  Communications in Mathematical Physics, 372 (2019), pp.~1--69.

\bibitem{KayKirkpatrick2011}
{\sc K.~Kirkpatrick, B.~Schlein, and G.~Staffilani}, {\em Derivation of the
  Two-dimensional Nonlinear Schr\"{o}dinger Equation from Many Body Quantum
  Dynamics}, American Journal of Mathematics, 133 (2011), pp.~91--130.

\bibitem{Li-20}
{\sc X.~Li and J.~Zhao}, {\em Orbital Stability of Standing Waves for
  Schr\"{o}dinger Type Equations with Slowly Decaying Linear Potential},
  Computers \& Mathematics with Applications, 79 (2020), pp.~303--316.

\bibitem{Li2021}
{\sc Y.~Li and F.~Yao}, {\em Derivation of the Nonlinear Schr\"{o}dinger
  Equation with a General Nonlinearity and Gross{\textendash}Pitaevskii
  Hierarchy in One and Two Dimensions}, Journal of Mathematical Physics, 62
  (2021), p.~021505.

\bibitem{Miao-22}
{\sc C.~Miao, J.~Zhang, and J.~Zheng}, {\em A Nonlinear Schr\"{o}dinger
  Equation with Coulomb Potential}, Acta Mathematica Scientia, 42 (2022),
  pp.~2230--2256.

\bibitem{rs2}
{\sc M.~Reed and B.~Simon}, {\em {Methods of Modern Mathematical Physics. {II}.
  {F}ourier Analysis, Self-adjointness}}, Academic Press [Harcourt Brace
  Jovanovich, Publishers], New York-London, 1975.

\end{thebibliography}

\def\cprime{$'$}

\end{document}